\documentclass[11pt,reqno]{amsart}

\usepackage{enumitem,kantlipsum}
\usepackage[table]{xcolor} 
\usepackage[table]{xcolor} 
\usepackage{booktabs} 
\usepackage{adjustbox} 
\usepackage{cancel}
\usepackage{epigraph}
\usepackage{geometry}                
\usepackage{subcaption}
\usepackage{moresize}
\setlist[enumerate, 1]{label=\arabic*.  }
\usepackage{float}
 \usepackage{pdfpages}
 \usepackage[
    backend=biber,   
    style=numeric,   
    url=false,
    doi=false,
    eprint=true,      
    maxbibnames=10,
]{biblatex}

\usepackage{cancel}

\usepackage[dvipsnames]{xcolor}
\usepackage{graphicx}
\usepackage{amssymb}
\usepackage{amsthm}
\usepackage{epstopdf}
\RequirePackage{doi}
\usepackage{hyperref}
\usepackage{graphicx}
\usepackage{forest}
\usepackage{tikz}
\usepackage{tikz-cd}
\usepackage{stmaryrd}
\usepackage{enumitem}
\usepackage{comment} 
\usepackage{newtxmath}
\usepackage{nicefrac}
\usepackage{pgfplots}
\usetikzlibrary{arrows,decorations.pathreplacing,calligraphy,patterns,positioning,matrix,calc,bending,decorations.pathmorphing,fit,backgrounds}

\pgfplotsset{compat=1.18}

\DeclareMathOperator{\diff}{diff}
\DeclareMathOperator{\sasc}{sasc}

\DeclareMathOperator{\lucky}{lucky}

\DeclareMathOperator{\Rec}{Rec}

\DeclareMathOperator{\psa}{psa}

\DeclareMathOperator{\probes}{probes}
\DeclareMathOperator{\wait}{wait}

\DeclareMathOperator{\Co}{Co}
\DeclareMathOperator{\rank}{rank}
\DeclareMathOperator{\corank}{corank}

\newcommand{\pfunction}[2][matrix]{%
  \begingroup
    \setlength{\arraycolsep}{3pt}
    \StrSubstitute{#2}{;}{\\}[\pfunc@tmp]%
    \StrSubstitute{\pfunc@tmp}{,}{ & }[\pfunc@tmp]%
    \begin{#1}
      \pfunc@tmp
    \end{#1}
  \endgroup
}

\definecolor{USred}{cmyk}{0,1.00,0.65,0.34}
\renewcommand{\emph}[1]{{\textcolor{USred}{\em #1}}}

\hypersetup{
  colorlinks,
  citecolor=USred,
  linkcolor=USred,
  urlcolor=USred}

\def\ZZ{\mathbb{Z}}

\def\NN{\mathbb{N}}

\def\SS{\mathbb{S}}

\theoremstyle{definition}
\newtheorem{de}{Definition}[section]
\theoremstyle{plain}
\newtheorem{thm}[de]{Theorem}

\newtheorem{lem}[de]{Lemma}
\newtheorem{p}[de]{Proposition}
\newtheorem{cor}[de]{Corollary}

\newtheorem{ex}[de]{Example}

\theoremstyle{definition}

\title{The priority lattice. }
\author{Adrián Lillo}
\author{Mercedes Rosas}

\address[A. Lillo, M. Rosas]{Departamento de Álgebra, Universidad de Sevilla}
\email[A. Lillo]{alillo@us.es}
\email[M. Rosas]{mrosas@us.es}

\begin{document}
\maketitle 

\begin{abstract}
 We introduce the priority lattice, a structure arising from the priority search algorithm on rooted trees and forests. We prove bijectively that its maximal chains are labeled by parking functions, and that the maximal chains of its principal ideals are labeled by partial parking functions. We establish that it is a graded lattice and compute its Möbius function and characteristic polynomials.
\end{abstract}

\epigraph{\em 
\flushright
ye can not see the wood for trees, \\ John Heywood, 1546}

\section{Introduction}

We introduce the \emph{priority lattice} as a unifying extension of two ideas: Richard P. Stanley’s  influential observation that maximal chains in the noncrossing partition lattice can be labeled by parking functions \cite{StanleyMaximalChainsParking}, and a bijection between parking functions and rooted trees that respect priority statistics, called the weary parking bijection and introduced in \cite{LRT-Weary}. As in the lattice of noncrossing partitions, maximal chains in the priority lattice are labeled by parking functions. Additionally, maximal chains of the order ideals are labeled by partial parking functions (where the number of cars and parking spots may differ). 

The nodes of the priority lattice are \emph{priority forests}, together with a top element that accounts for all failed parking processes. The priority forest of a partial parking function encodes the preferred spot of each car according to its final position. 
Parking functions with a fixed priority tree label chains passing through that tree, while partial parking functions with a fixed priority forest label chains in the principal order ideal with that forest as its top element.

This work begins by showing that the maximal chains of the priority lattice label ordered forests.
An ordered forest, or ordered $(m,n)$-forest to be more precise, is a rooted forest with $n + 1$ nodes and $m$ edges, together with a total order
 of its  component trees. We ask that its non-root vertices are labeled by $[m]$, while roots remain unlabeled and are distinguished by the index of their component tree. Cayley's forest formula says that the set of ordered $(m,n)$-forests has cardinality
    \begin{align*}
(n-m+1)(n+1)^{m-1}&\qquad &\text{\qquad(Cayley's forest formula,   \cite{ thebook,Cayley, LRT-bijective-visit}).}
    \end{align*}
A rooted tree with an unlabeled root can be viewed as an ordered forest with 
$n=m$, in this case, we recover Cayley’s formula for counting labeled trees.
In the setting of ordered forests, priority forests arise  from the use of the \emph{priority search algorithm}, an ordered forest traversal described in Section \ref{se:Priority_search_forest}.
The  priority search algorithm  interprets node labels as  priority ranks, always visiting the available node with highest priority.  During a priority search, all nodes of a component tree are visited before the search proceeds to the next tree. The \emph{priority traversal} of an ordered forest $F$ is the permutation that keeps track of the nodes as they are visited on a priority search. The priority forest of $F$ is the ordered forest resulting from relabeling the vertices of $F$ according to the order in which they are traversed. The component trees of a priority forest are always increasing trees.

Fix $n\ge0$. The priority lattice $\Pi(n)$ has as underlying set the collection of all priority forests labeled with  $\{0,1 \ldots, n\}$, together with an additional element $\hat{1}$. The partial order on  $\Pi(n)$ is defined by the containment of edge sets, with $\hat1$ explicitly set as the top element. 
The priority lattice is a graded lattice.
We endow $\Pi(n)$ with a partial \emph{edge labeling $\lambda$} that associates a permutation $\lambda(C)$ to any maximal chain $C$. Let  $P$ be a priority forest with $m$ edges. In Theorem~\ref{thm:maximal_chains_forests}, we define a bijection that sends a maximal chain $C$ of $[\hat0, P]$ labeled $\lambda(C)$ into an ordered $(m,n)$-forest with priority forest $P$ and priority traversal $\lambda^{-1}(C)$. Therefore, the maximal chains of $\Pi(n)$ are labeled by rooted trees, while the maximal chains of its proper intervals are indexed by ordered forests.

Cameron et al \cite{CameronJohannsenPrellbergSchweitzer2009} showed that there is another interesting set of objects counted by Cayley's forest formula, the set of \emph{\((m,n)\)-parking functions}, partial  parking functions with $m$ cars and $n$ spots. Let  $P$ be a priority forest with $m$ edges, in Theorem \ref{thm:maximal_chains_parking_functions}, we present a bijection between maximal chains of $[\hat0, P]$, and $(m,n)$-parking functions with priority forest  $P$ such that the bird's eye photo of the parked cars is $\lambda(C)^{-1}$.  
Combining Theorem \ref{thm:maximal_chains_forests} with Theorem  \ref{thm:maximal_chains_parking_functions},  we obtain a bijection between ordered \((m,n)\)-forests and \((m,n)\)-parking functions, the \emph{forest weary bijection}.
We describe this bijection explicitly, show that extends the weary bijection of \cite{LRT-Weary} to partial parking functions, and that it respects the priority  statistics of \cite{LRT-Weary}.
A translation of the main features of the priority lattice to the language of partial parking functions closes this study. Other bijections between partial parking functions and ordered forests present in the literature can be found in \cite{KenyonYin2023, kostic2006multiparkingfunctionsgraphsearching}.

At this point, we turn our attention back to the study of the priority lattice. We compute its corank generating function, and derive from it its Whitney numbers of the first and  second kind. Even if $\Pi(n)$ is neither distributive nor self-dual for $n \ge 3$, all intervals of $\Pi(n)$ that do not contain the top element are distributive for all $n\ge0$.

Let $P \le P'$ be priority forests. We show that the partial edge labeling of $\Pi(n)$ restricts to an EL-labeling of $[P,P']$. Then, in Lemma \ref{lem:Mobius_priority_forests}, we rest on a theorem of Stanley \cite{stanley1971modular} to compute $\mu(P, P')$. From here, we show in Corollary \ref{cor:Complete_Mobius}  that  $\mu(\hat0,\hat1) = 0$ for all $n\ge1$. 
On the other hand, upper intervals (intervals the form
$[P,\hat1]$ for some priority forest $P$) are not distributive. In Lemma~\ref{lem:mobius_upper_intervals}, we compute their Möbius function combining Lemma~\ref{lem:Mobius_priority_forests} with a sign-reversing involution. 
The work concludes with the computation of the characteristic polynomial of the priority lattice and its proper ideals. We show that, given an  $(m,n)$-priority forest $P$,
    \begin{align*}
  \chi([\hat0,P],q) =   q^{m-\sasc(P)} (q-1)^{\sasc(P)} = q^{m-\lucky(\pi)} (q-1)^{\lucky(\pi)}, 
  \end{align*}
   where $\sasc(P)$ is the number of small ascents in $P$. 
   The forest weary bijection implies the second equality, where  $\lucky(\pi)$ is the number of lucky cars  in any partial parking function $\pi$ with priority forest  $P$. This result is presented in Theorem   \ref{thm_charpoly}.
   Additionally, we also show that for any $n\ge1$, the characteristic polynomial of $\Pi(n)$ is 
    $
    q (q-1)^{n}
    $. 
    We close the paper with several open questions and final remarks.

\section{Basic definitions}
\label{se_basic_def}

A \emph{rooted tree} is a connected acyclic graph with a distinguished root, and a \emph{rooted forest} is a collection of rooted trees. A rooted tree with only the root is called trivial. An \emph{ordered $(m,n)$-forest} is a rooted forest $F$ with $n+1$ nodes and $m$ edges, whose component trees $T_0, T_1, \dots, T_{\,n-m}$ are totally ordered, with roots marked $\circ, \circ_1, \dots, \circ_{\,n-m}$ to record the order, and with its $m$ non-root vertices labeled by $[m]$.

 We assume that $n$ and $m$ always stand for nonnegative integers. We write $[n]$ for the set $\{1,2,\ldots,n\}$ and $[n]_0$ for the set $\{0,1,\ldots,n\}$. The interval of integers $[m,n]$ is the set of integers $z$ such that $m \le z \le n$. We consider these sets as ordered sets. Additionally, for any graph~$G$, we denote by $E(G)$ its edge set.

A \emph{partial function} $f : [m] \to [n]$ is a function from a subset $S$ of $[m]$ to $[n]$, called its \emph{domain}. The image of the elements of $[n] \setminus S$ under $f$ is undefined. As usual, we  identify $f$ with  the directed graph with vertex set $[m]\times [n]$ and a directed edge $(i,\omega(i))$ for each $i$ in its domain.
A partial function $f$ is said to be a \emph{partial permutation} when the restriction of $f$ to its domain is an injective map.
 In this situation, the \emph{inverse} of  $f$ is defined as  the partial permutation $f^{-1} : [n] \to [m]$ whose digraph has vertex set $[n]\times [m]$ and whose edge set is obtained by reversing the direction of all edges in the  digraph of $f$. 
 It is sometimes convenient to write partial functions using the usual two-line notation of permutations, and adding a symbol $-$ below those elements that do not belong to its domain. For simplicity, we usually omit the adjective ``partial'.

  \subsection{Priority forests and the priority first search algorithm}
\label{se:Priority_search_forest}

The \emph{priority-first search} is an ordered forest traversal that proceeds according to a priority function, giving higher priority to nodes with smaller labels. 
It is defined as follows.
Let $F = T_0 \ T_1 \ldots T_{n-m}$ be an ordered  $(m,n)$-forest.
Initially,  all nodes of $F$ are blocked except the children of $\circ$, the root of the first component tree. We proceed recursively. At each step, we visit the unblocked node with the highest priority, and unblock all of its children.  After all the nodes of a component tree have been visited, we move to the next one, read its root, unblock its root  children, and iterate this process. 
The resulting traversal  of $F$, where we record the roots by $-$, is referred to as its \emph{priority traversal}, and denoted by
\emph{$\omega_F$}.  Since all priority searches start at $\circ$, we do not include its visit in the priority traversal. We conclude that $\omega_F$ is the one-line word of a partial permutation of $[m]$.

Define the \emph{priority forest of $F$} to be the rooted forest obtained by relabeling the  nodes of $F$ according to the order they are visited during priority-first search, with $\circ$ labeled 0. The component trees are ordered according to the labels of their roots. 
More generally, a \emph{priority forest} (labeled with $[n]_0$ and with $m$ edges) is a rooted forest 
$(T_0, T_1, \ldots, T_{n-m})$ with all  component trees  increasing and ordered 
according to the root's labels, and where for all $j<k$, every label in $T_j$ 
is smaller than every label in $T_k$. Thus, the vertex set of each tree is an interval of integers; see Figure \ref{fig:priority forest}(B). Note that we do not call priority forests ``ordered forests'', since all their nodes, including the roots, are labeled. 

\begin{ex}
An ordered $(4,8)$-forest $F$  with priority traversal $\omega_F =   - \ - \ 4 \ 1 \ 3 \ - \ - \ 2$, and its priority forest $P$ are illustrated in Figure \ref{fig:priority forest}.  
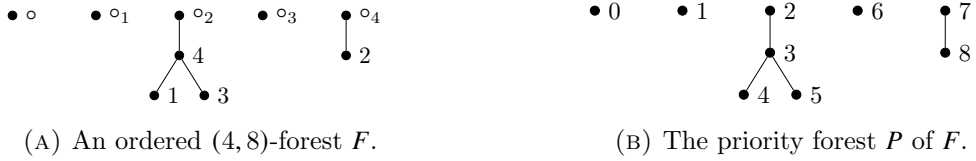
\begin{figure}[h!]
    \centering
    \begin{subfigure}[b]{0.48\textwidth}
        \centering
        \resizebox{0.75\textwidth}{!}{
        \normalfont \small
\begin{tikzpicture}[
    level distance=0.6cm,
    level 1/.style={sibling distance=0.75cm},
    level 2/.style={sibling distance=0.75cm},
    record/.style={
        circle,
        draw=black,
        fill=black,
        inner sep=1.3pt,
        label={#1},
        solid
    },
]
\def\rootsep{1.25}
    \node[record, label={right:$\circ$}] (1) at (0*\rootsep, 0) {};
    \node[record, label={right:$\circ_1$}] (1) at (1*\rootsep, 0) {};
    \node[record, label={right:$\circ_2$}] (2) at (2*\rootsep, 0) {}
        child {node[record, label={right:4}] (3) {}
            child {node[record, label={right:1}] (1) {}
                edge from parent[solid]}
            child {node[record, label={right:3}] (5) {}
                edge from parent[solid]}};
    \node[record, label={right:$\circ_3$}] (6) at (3*\rootsep, 0)  {};
    \node[record, label={right:$\circ_4$}] (7) at (4*\rootsep, 0) {} 
        child {node[record, label={right:2}] (8) {}
                edge from parent[solid]};
\end{tikzpicture}
        }
        \caption{An ordered $(4,8)$-forest $F$.}
        \label{fig:forest_F}
    \end{subfigure}
    \hfill
    \begin{subfigure}[b]{0.48\textwidth}
        \centering
        \resizebox{0.75\textwidth}{!}{\normalfont \small
\begin{tikzpicture}[
    level distance=0.6cm,
    level 1/.style={sibling distance=0.75cm},
    level 2/.style={sibling distance=0.75cm},
    record/.style={
        circle,
        draw=black,
        fill=black,
        inner sep=1.3pt,
        label={#1},
        solid
    },
]
\def\rootsep{1.25}
    \node[record, label={right:0}] (1) at (0*\rootsep, 0) {};
    \node[record, label={right:1}] (1) at (1*\rootsep, 0) {};
    \node[record, label={right:2}] (2) at (2*\rootsep, 0) {}
        child {node[record, label={right:3}] (3) {}
            child {node[record, label={right:4}] (1) {}
                edge from parent[solid]}
            child {node[record, label={right:5}] (5) {}
                edge from parent[solid]}};
    \node[record, label={right:6}] (6) at (3*\rootsep, 0)  {};
    \node[record, label={right:7}] (7) at (4*\rootsep, 0) {} 
        child {node[record, label={right:8}] (8) {}
                edge from parent[solid]};
\end{tikzpicture}
        }
        \caption{ The priority forest $P$ of $F$.}
        \label{fig:forest_PF}
    \end{subfigure}
    \caption{An ordered forest and its priority forest.}
    \label{fig:priority forest}
\end{figure}
\end{ex}

\section{The priority lattice}
\label{se_priority_lattice}

\subsection{The priority lattice}
\label{se:priority_lattice_def}
 In this section, we introduce the main construction of this article, the  priority lattice  $\Pi(n)$.
The underlying set of $\Pi(n)$ consists of all priority forests labeled with $[n]_0$, together with an extra element $\hat 1$. 
The  order  relation $\le$  on $\Pi(n)$ is defined as follows. Given two  priority forests $P$ and $P'$,  $P \le P'$  if and only if $E(P) \subseteq E(P')$. On the other hand,  $\hat 1$ is set to be the top element of $\Pi(n)$. Thus, the bottom element $\hat 0$ of $\Pi(n)$ is the priority forest that contains no edges. We conclude that  $\Pi(n)$ is a bounded poset.
The Hasse diagrams of $\Pi(2)$ and $\Pi(3)$ are presented in Figures \ref{fig:L2} and \ref{fig:Pi3}.

    \begin{figure}[h!]
    \centering
   \resizebox{0.3\textwidth}{!}{
   \tikzset{
    treenode/.style={
        circle,
        draw=black,
        fill=black,
        inner sep=0.65pt,
        outer sep=0pt,
        label={#1},
        solid
    },
    every label/.style={
            font=\tiny\fontsize{6}{7}\selectfont,
            label distance = 0.1pt,
            inner sep=1pt,
            outer sep=1pt,
        }
}
\begin{tikzpicture}[
    scale=0.7,
    level distance=0.3cm,
    level 1/.style={sibling distance=0.4cm},
    level 2/.style={sibling distance=0.4cm},
    pics/forest0/.style = {
        code = {
            \coordinate (#1-root) at (0,0);
            \matrix[matrix of nodes, row sep=0.3cm, column sep=0.1cm, ampersand replacement=\&]{
                \node[treenode, label={above:0}] (F0-T0-V-0) {}; \&
                \node[treenode, label={above:1}] (F0-T1-V-1) {}; \&
                \node[treenode, label={above:2}] (F0-T2-V-2) {}; \\
            };
        }
    },
    pics/forest1/.style = {
        code = {
            \coordinate (#1-root) at (0,0);
            \matrix[matrix of nodes, row sep=0.3cm, column sep=0.1cm, ampersand replacement=\&]{
                \node[treenode, label={above:0}] (F1-T0-V-0) {}; \&
                \node[treenode, label={above:1}] (F1-T1-V-1) {}
                    child { node[treenode, label={[overlay]right:2}] (F1-T1-V-2) {} }; \\
            };
        }
    },
    pics/forest2/.style = {
        code = {
            \coordinate (#1-root) at (0,0);
            \matrix[matrix of nodes, row sep=0.3cm, column sep=0.1cm, ampersand replacement=\&]{
                \node[treenode, label={above:0}] (F2-T1-V-1) {}
                    child { node[treenode, label={[overlay]right:1}] (F2-T1-V-2) {} }
                    child { node[treenode, label={right:2}] (F2-T1-V-3) {} }; \\
            };
        }
    },
    pics/forest3/.style = {
        code = {
            \coordinate (#1-root) at (0,0);
            \matrix[matrix of nodes, row sep=0.3cm, column sep=0.1cm, ampersand replacement=\&]{
                \node[treenode, label={above:0}] (F3-T0-V-0) {}
                    child { node[treenode, label={[overlay]right:1}] (F3-T0-V-1) {} }; \&
                \node[treenode, label={above:2}] (F3-T1-V-2) {}; \\
            };
        }
    },
    pics/forest4/.style = {
        code = {
            \coordinate (#1-root) at (0,0);
            \matrix[matrix of nodes, row sep=0.3cm, column sep=0.1cm, ampersand replacement=\&]{
                \node[treenode, label={above:0}] (F4-T0-V-0) {}
                    child { node[treenode, label={[overlay]right:1}] (F4-T0-V-1) {}
                        child { node[treenode, label={[overlay]right:2}] (F4-T0-V-2) {} } }; \\
            };
        }
    },
    pics/forest5/.style = {
        code = {
            \coordinate (#1-root) at (0,0);
            \node (one) at (0, 0) { $\hat 1$};
            };
        }
]

\def\xdist{2.25}
\def\ydist{2.75}
\pic (f0) at (0,0) {forest0};
\pic (f1) at (\xdist,\ydist) {forest1};
\pic (f3) at (-\xdist,\ydist) {forest3};
\pic (f4) at (\xdist,2*\ydist) {forest4};
\pic (f2) at (-\xdist,2*\ydist) {forest2};
\pic (f5)  at (0,3*\ydist) {forest5};

\begin{pgfonlayer}{background}

\draw[gray] (f0-root) -- node[midway, USred, font=\scriptsize, fill=white, inner sep=1pt] {2} (f1-root);
\draw[gray] (f0-root) -- node[midway, USred, font=\scriptsize, fill=white, inner sep=1pt] {1} (f3-root);
\draw[gray] (f3-root) -- node[midway, USred, font=\scriptsize, fill=white, inner sep=1pt] {2} (f2-root);
\draw[gray] (f1-root) -- node[midway, USred, font=\scriptsize, fill=white, inner sep=1pt] {1} (f4-root);
\draw[gray] (f3-root) -- node[midway, USred, font=\scriptsize, fill=white, inner sep=1pt] {2} (f4-root);
\draw[gray] (f2-root) -- (f5-root);
\draw[gray] (f4-root) -- (f5-root);

\def\circr{12mm}

\def\boxw{1.75cm} 
\def\boxh{1.5cm} 
\def\cornerr{10pt} 

\foreach \f in {f0,f1,f2,f3,f4}{
  \draw[gray,rounded corners=\cornerr, fill=white]
    ($(\f-root)+(-0.5*\boxw,-0.5*\boxh) - (0mm, 2mm)$) rectangle
    ($(\f-root)+(0.5*\boxw,0.5*\boxh)$);
  \coordinate (\f-mu) at ($(\f-root)-(0.65*\boxw, 1mm)$);
}
\draw[gray, rounded corners =0.65*\cornerr, fill=white] 
($(f5-root)+(-0.35*\boxw,-0.35*\boxh) - (0mm, 2mm)$) rectangle
    ($(f5-root)+(0.35*\boxw,0.35*\boxh)$);
\coordinate (one-mu) at ($(f5-root)-(0.5*\boxw, 1mm)$);
\end{pgfonlayer}

\node[green!50!black] at (f0-mu)  {\scriptsize $1$};
\node[green!50!black, xshift=-1mm] at (f1-mu)  {\scriptsize $-1$};
\node[green!50!black] at (f2-mu)  {\scriptsize $0$};
\node[green!50!black, xshift=-1mm] at (f3-mu)  {\scriptsize $-1$};
\node[green!50!black] at (f4-mu)  {\scriptsize $1$};
\node[green!50!black] at (one-mu)  {\scriptsize $0$};
\end{tikzpicture}
   }
    \caption{The Hasse diagram of $\Pi(2)$, its partial edge labeling, and its Möbius function values. }
    \label{fig:L2}
    \end{figure}
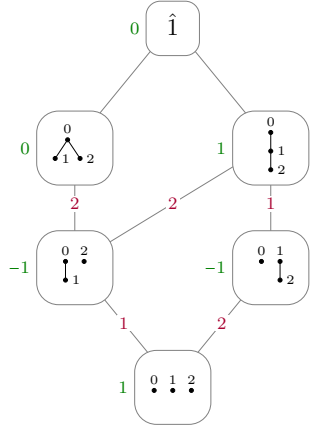

The \emph{cover relations} of $\Pi(n)$ can be easily described. If both $P$ and $P'$ are  priority forests and $P'$ covers $P$, written \emph{$P \lessdot P'$}, when $P'$ can be obtained from $P$ by adding a single edge to one of its component trees. On the other hand, $\hat 1 $ covers precisely the set of priority trees. Priority trees are always increasing trees and are in bijection with permutations in $\SS_n$.

The \emph{meet} and the \emph{join} on the priority lattice can be described as follows.    For any element  $x$ of $\Pi(n)$,  $\hat 1 \wedge x = x$, and $\hat 1 \vee x = \hat 1.$ On the other hand, if both $P$ and $P'$ are priority forests, 
           \begin{align*}
             E(P\wedge P')  =  E(P) \cap E(P') \qquad \text{ and } \qquad 
               E(P\vee P')  =  E(P) \cup E(P').
           \end{align*}

  \begin{figure}[h!]
    \centering
    \resizebox{0.62\textwidth}{!}{\input{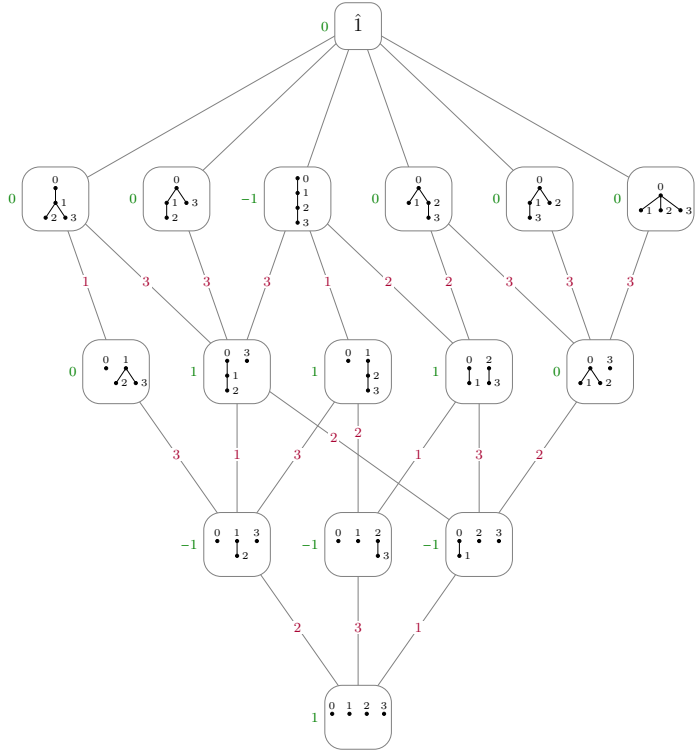}}
    \caption{The Hasse diagram of $\Pi(3)$, its partial edge labeling, and its Möbius function values.}
    \label{fig:Pi3}
\end{figure}

      \begin{lem} \label{Phi_graded}
          The priority lattice $\Pi(n)$ is a graded lattice.
          \end{lem}

    \begin{proof} We need to show that all maximal chains of $\Pi(n)$ have length $n+1$.   Assume $n>1$, and fix a maximal chain. The maximal forest of the chain has to be a tree, otherwise it would have at least two consecutive component trees, and we could graft the tree on the right to the root of the tree on the left and refine the chain. 
    On the other hand, if $P  \lessdot P'$ are two forest  in the chain, $E(P')\setminus E(P)$ has to consist of one edge. Otherwise, there would be a forest between them, the one obtained attaching either edge to $P$. 
    Finally, $\hat 1$ belongs to any maximal chain. Thus, the maximal chain has  length $n+1$. 
          \end{proof}

    The  \emph{rank function} $\rho : \Pi(n) \to \ZZ_{\ge0}$ is given by
        \[
        \rho(P) = 
        \begin{cases}|E(P)| &\text{if $P$ is a priority forest,}\\
        n+1 &\text{if $P= \hat 1$.}
        \end{cases}
        \]
       The rank $m$ of a priority forest $P$ equals its number of edges, whereas the \emph{corank} $n +1 - m$ equals the number of component trees of $P$.

The \emph{atoms} of $\Pi(n)$ (rank one elements) are the $(1, n)$-priority forests. These forests contain a unique small ascent, an  edge $\{i, i+1\}$ for some $i\leq n-1$.
On the other hand, the \emph{coatoms} of $\Pi(n)$, which are the rank $n$ elements, are precisely the increasing trees.


\subsection{The maximal chains of the priority lattice.}
\label{su:max_chain}

We show that  the maximal chains of   $I_P = [\hat 0, P]$, for a fixed priority forest $P$, are in bijection with  the set of ordered $(m,n)$--forests with  priority forest $P$. 
Assume that  $P$ and $P'$ stand for two priority forests in $\Pi(n).$
A  partial function $\lambda : E \to \NN,$ from the set of covering edges of the Hasse diagram of $\Pi(n)$ to the natural numbers is referred to as a \emph{partial edge labeling}.  

 If  $P' \lessdot P$, let $(P', P)$ be the edge representing this cover relation in the Hasse diagram of $\Pi(n)$. 
The \emph{label $\lambda(P', P)$} of $(P', P)$ is set to be the larger endpoint of the unique edge in  $E(P)\setminus E(P')$. 
  We do not assign any  label to the  cover edges involving $\hat 1$. 
  The Hasse diagrams of $\Pi(2)$ and $\Pi(3)$ appearing in Figures \ref{fig:L2} and \ref{fig:Pi3} include the partial edge labelings.

  Let $C :  P_0 \lessdot P_1 \lessdot \ldots \lessdot P_{k-1}\lessdot  P_k=P$ be  a maximal chain of  $[P_0, P]$. The labeling
\[
\lambda(C) = (\lambda(P_0, P_1), \lambda(P_1, P_2), \dots, \lambda(P_{k-1}, P))
\]
is always a partial permutation because the existence of a repeated letter on $\lambda(C)$ would translate into a cycle in $P$. 
We refer to $\lambda(C)$ as a \emph{Jordan-Hölder permutation} of $C$. 

Denote by $F_C(P)$ the ordered forest $\lambda(C)^{-1}(P)$ obtained after relabeling the non-roots of $P$ according to the partial permutation $\lambda(C)^{-1}$, and recording the relative order of the roots in the subscripts of the $\circ$'s. In Lemma \ref{lem: Jordan Holder => priority traversal}, we totally characterize $F_C(P)$.

\begin{lem}
\label{lem: Jordan Holder => priority traversal}
Let $P$ be a priority forest, and let $C$ be a maximal chain of $[\hat 0, P]$ with Jordan-Hölder permutation $\lambda(C)$. Then, $F_C(P)$ is the unique ordered forest with priority traversal $\lambda(C)^{-1}$ and priority forest $P$.
\end{lem}

\begin{proof}
Note that if $\lambda(C)^{-1}(P)$ 
has priority traversal $\lambda(C)^{-1}$, then $\lambda(C)^{-1}(P)$ has priority forest $P$ by definition.
We prove that $F_C(P)$ has priority traversal $\lambda(C)^{-1}$ by induction on the rank of $P$. 
If $\rank(P) = 0$, both $\lambda(C)^{-1}$ and 
    the priority traversal of $\lambda^{-1}(C)(P)$ are the partial permutation with empty domain.
    Suppose that the result holds for all priority forests of rank $m - 1$, and suppose that $P$ has rank $m$. Let $C'$ be the chain obtained from $C$ removing its last element, and let $P'$ be the top element of $C'$. Let $\tau$ and $\tau'$ denote the priority traversals of $\lambda(C)^{-1}(P)$ and $\lambda(C')^{-1}(P')$, respectively.
    By induction hypothesis, $\tau' = \lambda(C')^{-1}$.

    The Jordan-Hölder permutations of $C$ and $C'$ coincide everywhere except at $m$, as $\lambda(C')(m)$ is undefined, whereas $\lambda(C)(m) = i$ for certain $i$ in $[n]$. 
    On the other hand, $\lambda(C)^{-1}(P)$ can be obtained from $\lambda(C')^{-1}(P')$ by
    relabeling the root 
    that is found at the $i$-th step of the traversal with the label $m$, and determining a parent for $m$  in the tree to its left. 
    Furthermore, we have 
    $
    \tau = \tau'(1) \tau'(2) \cdots \tau'(i-1) \ m \ \tau'(i+1) \cdots \tau'(n). 
    $
    To see this, note that $\tau$ coincides with $\tau'$ up to the moment when vertex $m$ is visited. Indeed, at each earlier step, the unblocked vertices in $\lambda(C)^{-1}(P)$ coincide with those on $\lambda(C')^{-1}(P')$, possibly with the addition of $m$.  However, since $m$ has the lowest priority, it is never selected as long as other candidates are available.  After $i-1$ steps, the only unvisited vertices on $\lambda(C)(P)$ are $m$, its descendants, and those in trees to the right of $m$. At the $i$-th step, vertex $m$ is visited. From here, priority search on $\lambda(C)(P)$ proceeds exactly as on $\lambda(C')(P')$, as at every step the set of unblocked vertices of both forests coincide (this time without exceptions).

    Combining this with $\tau' = \lambda(C')^{-1}$, we conclude that  $\tau = \lambda(C)^{-1}$, as desired.
    Uniqueness is immediate, as an ordered forest is uniquely determined by its priority forest together with its priority traversal.
    \qedhere
\end{proof}

The \emph{principal ideal}  of a priority forest $P$   is defined as $I_P = \{P' \in \Pi(n)  \ | \ P' \le P \}$. We do not consider $\Pi(n)$ to be a principal ideal.

\begin{lem}
\label{lem:priority traversal => Jordan Holder}
    Let $F$ be an ordered  with priority forest $P$ and priority traversal $\tau$. Then, $\tau^{-1}$ is the Jordan-Hölder permutation of a unique maximal chain of $I_P=[\hat 0, P]$.
\end{lem}
\begin{proof}
     Set $m$ to be the rank of $P$. For each $i$ in  $[m]$, let $P_i$ be the forest whose edges are those joining  $\tau^{-1}(j)$ to its parent in $P$ for $j \leq i$.
   We claim that the sequence of forests 
    $
    C :\hat 0 \subseteq P_1 \subseteq \cdots \subseteq P_{m} = P
    $
    is a maximal chain of $[\hat 0, P]$. The length of $C$ is $m$, and inclusions are immediate from the construction. It therefore remains to show that each $P_i$ is a priority forest.
       We proceed by downward induction on $i$.
       
       The case $i = m$ is trivial, as $P_m = P$ is a priority forest by assumption. Suppose  that $P_{i}$ is a priority forest and consider $P_{i - 1}$.
    The forest $P_{i-1}$ is obtained from $P_{i}$ by deleting the edge $e$ joining $\tau^{-1}(i)$ to its parent. Let $T$ be the tree of $P_{i}$ containing $\tau^{-1}(i)$, and suppose that its vertex set is the interval $[a, b]$.
    Removing $e$ splits $T$ in two trees: a tree $T_1$ containing all descendants of $\tau^{-1}(i)$, and a tree $T_2$, containing the remaining vertices. Every tree of $P_{i - 1}$ besides $T_1$ and $T_2$ is also a tree of $P_{i}$, and hence has an interval vertex set. 
    We claim that the vertex sets of $T_1$ and $T_2$ are, respectively,  
    $[\tau^{-1}(i), b]$ and $[a, \tau^{-1}(i) - 1].$
    This will show that $P_{i-1}$ is a priority forest. 

    Since $T$ is increasing, every vertex in $T_1$ is greater than or equal to  $\tau^{-1}(i)$ (the root of $T_1$) and it is left to show that every vertex of $T_2$ is less than $\tau^{-1}(i)$. Suppose otherwise, and let $x$ be a vertex of $T_2$ with $x > \tau^{-1}(i)$. By definition, the roots of $P_{i}$ are the vertices $\tau^{-1}(k)$ for $k > i$. Note that $x$ is not the root of $T$, as $x > \tau^{-1}(i)$. Therefore, $x$ equals $\tau^{-1}(j)$ for some $j < i$. 
    Since $\tau^{-1}(j) > \tau^{-1}(i)$, vertex $j$ is visited after $i$ in the priority traversal of $F$, despite $j$ having higher priority.  Thus, $j$ must have an ancestor $a > i$ that is visited after $i$. This is, such that $\tau^{-1}(a) > \tau^{-1}(i)$. This implies that $\tau^{-1}(a)$ is a root of $P_{i}$, and hence the edge of $F$ joining $\tau^{-1}(a)$ to its parent does not belong to $P_{i}$. However, this edge needs to be crossed to move from $\tau^{-1}(j)$ to $\tau^{-1}(i)$ in $P$, since $\tau^{-1}(a)$ is an ancestor of $\tau^{-1}(j)$ greater than $\tau^{-1}(i)$ and $P$ is increasing.
    It follows that that $x = \tau^{-1}(j)$ and $\tau^{-1}(i)$ do not belong to the same tree of $P_{i}$, as there is no path between them. This contradicts our choice of $x$.

    Finally, the Jordan-Hölder permutation of the chain $C$ is $\tau^{-1}$ by construction. Uniqueness follows from the fact that a maximal chain of $[\hat 0, P]$ is completely determined by its Jordan-Hölder permutation. \qedhere
\end{proof}

We are ready to present the main result of this section.
   Define a \emph{complete $m$-chain} as a saturated chain starting at $\hat0$ of length $m.$  
\begin{thm} 
\label{thm:maximal_chains_forests}
Fix $m\le n$. There is a unique bijection $\phi_m$ between the set of complete $m$-chains of $\Pi(n)$ and the set of ordered $(m,n)$-forests that sends a chain $C$ with top element $P$ and Jordan-Hölder permutation $\lambda$ to a ordered forest with priority forest $P$ and priority traversal $\lambda^{-1}$. The bijection $\phi_m$ is explicitly defined as
\[
\phi_m(C) = \lambda^{-1}(P).
\]
\end{thm}

\begin{proof}
Given a ordered $(m,n)$-forest $F$ with priority forest $P$ and priority traversal $\tau$, define $\overline\phi_m(F)$ to be the unique maximal chain of $I_P=[\hat 0, P]$ with edge-labelling $\tau^{-1}$. The existence of this maximal chain is guaranteed by Lemma \ref{lem:priority traversal => Jordan Holder}. Note that $\overline\phi_m(F)$ is a complete $m$-chain of $\Pi(n)$ with top element $P$. By Lemma \ref{lem: Jordan Holder => priority traversal}, the maps $\phi_m$ and $\overline\phi_m$ are inverses of each other, and hence $\phi_m$ is a bijection.
\end{proof}

 \begin{figure}[h!]
    \centering
    \input{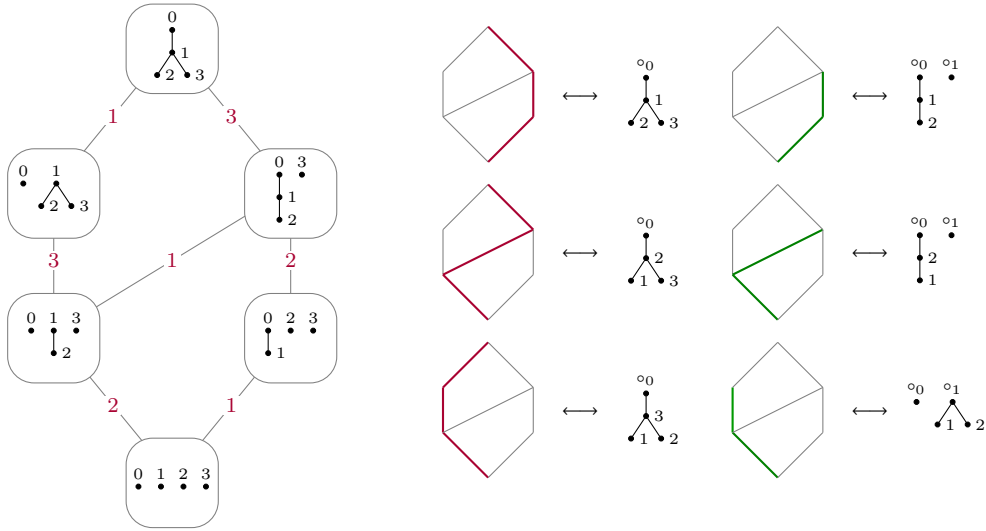}
      \caption{An interval $[\hat{0}, T]$, its maximal chains and corresponding rooted trees, and  some additional examples at the forest level. }
    \label{fig:Interval_Pi3}
\end{figure} 

As a result, for fixed  $m \le n$. The number of complete $m$-chains of $\Pi(n)$ is given by Cayley's forest formula, $(m-n+1)(n+1)^{m-1}$, and  maximal chains of $\Pi(n)$ are enumerated by Cayley's tree formula, $(n+1)^{n-1}$.
In Figure \ref{fig:Interval_Pi3} we present an example of this bijection for trees.
The three maximal chains with $T$ as top element are illustrated at its right, together with the rooted tree obtained by acting on $T$  with the inverse of the respective Jordan–Hölder permutation. Observe that if we read the edge labels of the maximal chains of $[\hat0, T]$, from bottom to top, 2 always appears before 3.


\section{Partial parking functions.}
\label{Part:parking_functions}

 Consider $n$ cars arriving in order to a street with  $n$ parking slots, in which each car has a preferred spot $a_i$. When Car $i$ enters the street, it first attempts to park at spot $a_i$. If that spot is unoccupied, it parks; otherwise, it proceeds forward and parks at the next available free spot. The sequence of preferences $\pi = (a_1, a_2,\ldots, a_n)$ is a \emph{parking function} of length $n$ when all cars succeed to park. Otherwise, we refer to it as a failed preference sequence.
Parking functions first appeared in the study of linear probes
of random hashing functions in the work of Konheim and  Weiss,  \cite{KW}. There, the authors famously showed that  there are 
$
(n+1)^{n-1}
$
parking functions of length $n$.
After Cayley \cite{Cayley}, we know that this number also counts the number of  trees of order $n+1.$  

  Konheim and Weiss did not assume that the number of cars and the number of parking spots were equal until the  very last section of their work. Parking functions with fewer cars than parking slots will be one of the central themes of this work.
    An \emph{$(m,n)$-parking function} is a preference sequence $\pi : [m] \to [n]$ where cars labeled with $[m]$ arrive in increasing order to a street with $n$ parking slots, labeled with $[n]$. 
  When car $i$ arrives, it attempts to park at spot $\pi(i)$. If the spot is empty, it parks there; otherwise, it continues driving until it finds an empty parking spot and parks there. We say that $\pi$ is an $(m,n)$-parking function if all $m$ cars are successfully able to park in the $n$ spots. 
  This immediately implies that $m\le n$.
  
  We also refer to $(m,n)$-parking functions as \emph{partial parking functions}.  Note that when the number of cars coincides with the number of spots, we recover the usual definition of parking functions. 
  Partial parking functions are called occupancy disciplines in \cite{KW}, and generalized parking functions in \cite{KenyonYin2023}.   
 
Let $O(\pi)$ be the set of occupied spots under $\pi$. 
The \emph{bird's eye permutation} of $\pi$ is the partial permutation  $\omega_\pi :  [n] \to [n]$ that sends a spot $s_i\in O(\pi)$ into $\omega_\pi(s_i)$, the label of the car occupying $s_i$, and leaves undefined the image of all $s_i\not\in O(\pi)$.
When written in one-line notation, the bird's eye permutation gives the bird's eye picture of the parking lot after the parking process defined by $\pi$ has ended, with ``$-$" standing for empty spots. The (partial) inverse permutation $\omega_\pi^{-1}$ is referred to as the \emph{outcome permutation} of $\pi$.
 
 \begin{ex} 
\label{ex:ppf}
Consider the parking function $\pi = (2, 4, 2, 1, 3)$.
Let $\pi^k$ be the $(5,k)$-parking function obtained by restricting
  $\pi$ to the first $k$ cars, and let $\omega^k$ be the partial bird's eye permutation of $\pi^k$.
The restriction of $\pi$ to $\emptyset$ is the empty partial parking function $\pi^0$, so no cars park. The restriction $\pi^1=(2)$ shows that Car~1 parks at spot~2, while $\pi^2=(2,4)$ indicates that Car~2 parks at spot~4. For $\pi^3=(2,4,2)$, Car~3 prefers spot~2 but finds it occupied and therefore parks at spot~3, the next available position. The restriction $\pi^4=(2,4,2,1)$ shows that Car~4 parks at spot~1. Finally, $\pi=(2,4,2,1,3)$ is a parking function, and Car~5, finding spot~3 occupied, parks at spot~5.
 The resulting sequence of bird's eye permutations of this sequence  is
  \begin{multline*}
  \omega^0=  \begin{pmatrix}
1 & 2 & 3 & 4 & 5\\
- & - & - & - & -
\end{pmatrix} \qquad
 \omega^1 = \begin{pmatrix}
1 & 2 & 3 & 4 & 5\\
- & 1 & - & - & -
\end{pmatrix}\qquad
\omega^2=
\begin{pmatrix}
1 & 2 & 3 & 4 & 5\\
- & 1 & - & 2 & -
\end{pmatrix} \\
\omega^3=
\begin{pmatrix}
1 & 2 & 3 & 4 & 5\\
- & 1 & 3 & 2 & -
\end{pmatrix}\qquad
\omega^4=
\begin{pmatrix}
1 & 2 & 3 & 4 & 5\\
4 & 1 & 3 & 2 & -
\end{pmatrix} \qquad
\omega^5=
\begin{pmatrix}
1 & 2 & 3 & 4 & 5\\
4 & 1 & 3 & 2 & 5
\end{pmatrix}
\end{multline*}
 \end{ex}

     A useful way of representing parking functions is by the use of  parking blueprints, \cite{LRT-Weary}. We extend them to partial parking functions. Let $\pi$ be a $(m,n)$-parking function, and let $\omega =\omega_\pi$ be its bird's eye permutation. 
The \emph{parking blueprint} of $\pi$ has $n$ consecutive segments representing the parking spaces. For each $ i$ in $[m]$, as Car $i$ parks at space $\omega(i)$, we draw Car $i$ on top of segment $ \omega(i)$. Additionally, we draw an arc connecting $\pi(i)$, the preferred position of Car $i$, to  $\omega(i)$, the position where it parks. We do so taking care that arcs do not intersect. For simplicity, if a car parks in its preferred space, we draw no arc. The central column of Figure \ref{Fig:partial_parking_cars} 
shows the sequence of parking blueprints corresponding to the sequence $\pi^0, \pi^1, \dots, \pi^5$ of Example \ref{ex:ppf}.
\begin{figure}
    \centering
    \resizebox{!}{0.55\textheight}{
    \input{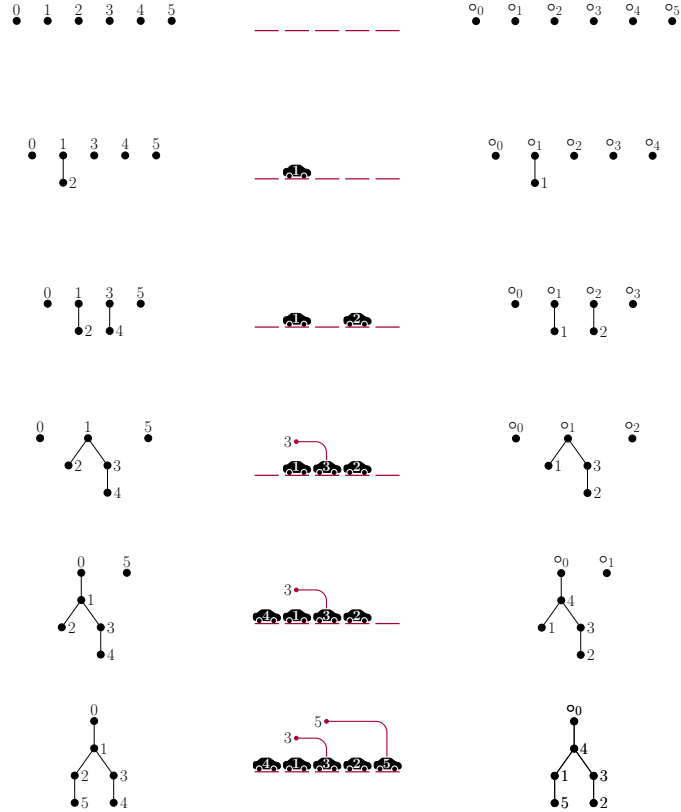}}
    \caption{In the central column, a parking function as a sequence of partial parking functions. In the left column, the corresponding sequence of priority forests. In the right column, the corresponding sequence of ordered forests under the weary bijection.}
\label{Fig:partial_parking_cars}
\end{figure}

Both the partial parking function $\pi$ and its bird's eye permutation $\omega_\pi$ can be easily recovered from the parking blueprint.  Scanning the $n$ segments from left to right yields $\omega_\pi$, recording a $-$ for each unoccupied space. To reconstruct $\pi$, one reads the arcs in reverse: if an arc terminates at Car $i$, its starting point indicates the preferred space $\pi(i)$. If no arc terminates there, Car $i$ is parked at its preferred space.

\subsection{The priority forest of a partial parking function} We define the priority tree of a partial parking function, and use this notion to describe the maximal chains of the priority lattice and of its proper intervals in the language of partial parking functions.

Let $P$ be a priority forest. The partial map sending each non-root node of $P$ to its parent, and leaving the image of roots undefined is known as the \emph{parent map}. It is a partial function $p : [n] \to [n - 1]_0$, as vertex $n$ has no children.  

The \emph{shifted parent map} of $P$, denoted by $s_P$, is the partial function $s_P:[n] \to [n]$ given by $s_P(i) = p(i) + 1$ whenever $p(i)$ is defined.
A shifted parent map verifies $s_P(i) \leq i$ for every $i$ in its domain because priority forests are increasing. When $P$ is a tree, shifted parent maps are subexceedant functions, see \cite{LRT-Weary}. 
When a forest is not a tree, the domain of its shifted parent map is strictly contained in $[n]$. In general, however, it need not be $[m]$ for some $m < n$. Thus, shifted parent maps are not always partial parking functions. Nevertheless, every partial parking function induces a shifted parent map, as shown by the following lemma.

\begin{lem}
\label{lem:shifted_parent_map}
    Let $\pi$ be an $(m, n)$-parking function with bird's-eye permutation $\omega$. Then $\pi \circ \omega$ is the shifted parent map of a priority  forest with $m$ edges.
 \end{lem}
\begin{proof} Fix $i$ in $ [n]$, and assume that there is a car parked at spot $i$.  By definition, this car has preference $\pi\circ \omega(i)$, and $\pi \circ \omega(i) \leq i$, as the preference of a car never exceeds its final parking spot. Thus, $\pi \circ \omega$ is the shifted parent map of an increasing forest $P$. 
    
    Suppose that $P$ is not a priority forest.
     Then there exists a vertex $x$ that is larger than the root $r$ of some tree to its right. Choose $x$ minimal with this property. By the minimality of $x$, its parent must be strictly less than $r$.
     Thus, we have $\pi \circ \omega (x) \leq r$, and $\pi \circ \omega (r)$ is undefined. This implies that car $\omega(x)$ drives past spot $r$ while it is empty, despite its preference being less than or equal to $r$. Contradiction. Thus, $P$ is a priority forest. 
\end{proof}

The \emph{priority forest} of an $(m, n)$-parking function $\pi$ is defined as the priority forest described in the statement of Lemma \ref{lem:shifted_parent_map}. The shifted parent map of the priority forest of $\pi$ can also be read from the parking blueprint: the starting point of an arc ending at spot $i$ gives $s_P(i)$, and an occupied spot $i$ with no incoming arc implies $s_P(i) = i$. If spot $i$ is unoccupied, then $s_P(i)$ is undefined, and hence $i$ is a root of $P$. 

\begin{lem}
\label{lem:bird's eye => Jordan Holder}Let $C$ be a maximal chain of $[\hat 0, P]$ with Jordan-Hölder permutation $\lambda$. Then, $s_P \circ \lambda $ is a partial parking function with priority forest $P$ and bird's eye permutation $\lambda^{-1}$. 
\end{lem}

\begin{proof}
    We claim that, if $s_P \circ \lambda$ has bird's eye permutation $\lambda^{-1}$, then  $s_P \circ \lambda$ has priority forest $P$. Indeed,  $(s_P \circ \lambda) \circ \lambda^{-1} = s_P \circ \epsilon_P = s_P$, where $\epsilon_P$ denotes the identity permutation of the non-roots of $P$. 
We prove that $s_P \circ \lambda$ has bird's eye permutation $\lambda^{-1}$ by induction on the rank of $P$. 

If $\rank(P) = 0$, the result is immediate. Suppose that the result holds for all priority forests of rank $m - 1$, and suppose that $P$ has rank $m$. Let $C'$ be the chain obtained from $C$ removing its last element, and let $P'$ be the top element of $C'$. Let $\omega$ and $\omega'$ denote the bird's eye permutations of $s_P \circ \lambda(C)$ and $s_P' \circ \lambda(C')$, respectively.
    By induction hypothesis, $\omega' = \lambda(C')^{-1}$.
    Let $i = \lambda(C)(m)$. 
    Note that $s_P \circ\lambda(C)$ can be obtained from $s_P' \circ\lambda(C')$ by setting the preference of car $m$ to be equal to $s_P(i)$. Furthermore, we claim that $\omega = \omega'(1) \cdots \omega'(i-1) m \omega'(i+1) \cdots \omega'(n)$. The parking procedure is identical for the first $m-1$ cars of $s_P \circ \lambda(C)$ and $s_{P'} \circ \lambda(C')$, so we only need to show that the parking spots in the interval $[s_P(i), i - 1]$ are occupied when car $m$ attempts to park.
    For contradiction, suppose that a space $x$ in the interval $[s_P(i), i - 1]$ is unoccupied. Then $\lambda(C')^{-1}(x)$ is undefined, and hence $x$ is a root of $P'$. The only root of $P'$ that is not a root of $P$ is $i$, so $x$ is also a root of $P$. Since $s_P(i) \leq x < i$, this contradicts the fact that $P$ is a priority forest, as the tree containing $i$ and its parent $s_P(i) - 1$ does not have an interval vertex set.
    Equality $\omega = \lambda(C)^{-1}$ then follows by the induction hypothesis.
\end{proof}

Since a parking function is determined by its priority forest and its bird's eye permutation, the map given by the above lemma is injective. On the other hand, as the domain and codomain have the same cardinality, we have defined a bijection. The inverse map sends a parking function $\pi$ to the unique chain with top element the priority forest of $\pi$, and with Jordan-Hölder permutation $\omega^{-1}_\pi$.
\begin{thm} 
\label{thm:maximal_chains_parking_functions}
Fix $m$ in $[n]$. There is a unique bijection $\psi_m$ between the set of complete $m$-chains of $\Pi(n)$ and the set of $(m,n)$-parking functions that sends a chain $C$ with top element $P$ and Jordan-Hölder permutation $\lambda$ to a partial parking function  
with priority forest $P$ and bird's eye permutation $\lambda^{-1}$. The bijection $\psi_m$ is defined as
\[
\psi_m(C) = s_P \circ \lambda,
\]
where $s_P$ is the shifted parent map of $P$.   
\end{thm}

\subsection{The forest weary parking bijection}
\label{se:weary_bijection}

Complete $m$‑chains of the priority lattice are in bijection with both ordered $(m,n)$‑forests and $(m,n)$‑parking functions. This correspondence naturally induces a bijection between these two families, which, in the case $m = n$, recovers the weary parking bijection of \cite{LRT-Weary}.
A different bijection between ordered forests and partial parking functions can be found in \cite{KenyonYin2023}.

We summarize the tree \emph{weary parking} process of \cite{LRT-Weary}.
   Let $T$ be a tree labeled with $[n]_0$ and rooted at $\circ$, that we identify with $0$. We interpret $T$ as a bird's eye view of the relative positions of $n$ cars: the non-root vertices represent the cars, the edges represent streets, and the root $\circ$ represents the entrance to a street with plenty of parking space.
    Labels are used to encode the speed ranking of cars, with Car $1$ being the fastest car, and Car $n$ the slowest. We assume that overtaking is forbidden. At every step of the parking process, the fastest unblocked car reaches the entry point $\circ$ first and proceeds to park at the nearest available spot. This unblocks all its children, and the procedure is iterated until all cars have finally parked. Hence, the order in which cars enter the parking street and park coincides with the priority traversal $\tau$ of $T$. By construction, all cars end up finding a parking space.

The \emph{weary parking bijection} is constructed as follows.
 As overtaking is forbidden, the best a driver can hope for is to park immediately after the car that is blocking it in $T$, so we define this spot to be its preferred spot. Explicitly, given a car $i$ with parent $p$, the \emph{preference} of Car $i$ is defined to be $\pi_T(i) = \tau^{-1}(i) + 1$ (this is, the parking spot of Car $i$, plus one). 
  In \cite{LRT-Weary} it was shown that $\pi_T$ is a parking function and that the map sending a rooted tree $T$ to the parking function $\pi_T$ is a bijection. In particular, it was shown that the bird's eye permutation of $\pi_T$ coincides with the priority traversal of $T$.

On the other hand, cars in partial parking function park in \emph{car clusters},  maximal collections of cars with no empty spots between them. Therefore,  the forest weary parking process must allow for multiple entry points to the parking street. Given an ordered $(m,n)$-forest $F=(T_0, T_1,\dots,T_{n-m})$, we define the \emph{entry points $e_k$}, for each $k$ in $[n-m]_0$, as
\[
e_k = 1 +\sum_{j=0}^{k-1} t_j, 
\]
where $t_j = |T_j|$ denotes the number of nodes of $T_j$ (counting the root).

The \emph{forest weary parking} process defined by an ordered $(m,n)$ forest $F$ involves $m$ cars, $n$ parking spots and $n-m+1$ car clusters, and empty spots $\{e_1 - 1, e_2 - 1, \cdots, e_{k} - 1\}$. Cars enter the street according to a priority search on $F$ and park at the nearest available spot, with the additional restriction that cars in $T_i$ must enter the street at the entry point $e_i$. 
 The forest weary parking process is always successful. When $m = n$, we recover the tree weary parking process. 
Let  $\omega_F$ be be the partial bird's eye permutation of $F$, and let $p_F$ be its partial parent function.
  The \emph{preference function} $\pi_F$ of $F$ is defined as 
\begin{align}
\label{def:preference}
\pi_F(i) = \omega_F^{-1}(p_F(i)) + 1
\end{align}  
if $p_F(i)$ is not  a root. Otherwise, $\pi_F(i)$ is left undefined.  
Observe that for all $F$, we have $\pi_F = \psi_m\circ\phi_m^{-1}(F)$. The “plus one” appearing in the definition of $\pi_F$ (see Eq.~\ref{def:preference}) is a consequence of using the shifted parent map in Theorem \ref{thm:maximal_chains_parking_functions}. In particular, $\pi_F$ is an $(m, n)$-parking function. 
 Let $\mathcal{F}_{m,n}$ be the set of ordered $(m,n)$-forests and let $\mathcal{P}\mathcal{F}_{m,n}$ be the set of $(m,n)$--parking functions.
We define the \emph{forest weary bijection} as 
\begin{alignat*}{2}
  &\rho_{m} : \mathcal{F}_{m,n} \to \mathcal{P}\mathcal{F}_{m,n}  \\
&\ \ \phantom{\rho_{m,n} :} F \mapsto  \pi_F 
\end{alignat*} 
 that sends an ordered $(m,n)$-forest $F$ to its preference function $\pi_F$. 
 The factorization $\rho_m = \psi_m \circ \phi_m^{-1}$, together with Theorem \ref{thm:maximal_chains_forests} and Theorem \ref{thm:maximal_chains_parking_functions}, implies that $\rho_{m}$ is a bijection. For $n = m$, this construction recovers the tree weary bijection from \cite{LRT-Weary} between rooted trees of order $n+1$ and parking functions of length $n$.

\subsection{Statistics and the forest weary bijection.}  
The tree weary bijection preserves a  family of \emph{priority statistics}. 
On the tree side, these statistics arise from the priority traversal, while for parking functions 
they correspond to classical numerical invariants, such as the number of linear probes or the number
of lucky cars. On top of this, the set of records of $T$ (that is, vertices with a greater label than 
all its ancestors) coincides with the set of records of $\pi_T$, defined as left-right maxima of the 
bird's eye permutation of $\pi_T$. These equidistribution results were established on \cite{LRT-Weary}. In this section, we show that they can be generalized to the setting of partial parking functions and ordered forests, and much information about the statistics is neatly 
encoded in the priority lattice. 

First, observe that any complete chain of the priority lattice is determined by its top element (a priority forest) and its Jordan-Hölder permutation (a partial permutation). 
We consider statistics on both of these objects. 
Let $P$ be a priority forest. A non-root vertex $i$ of $P$ is called a \emph{small ascent} if its parent is $i - 1$. We denote by \emph{$\sasc(P)$} the number of small ascents of $P$.
    Let \emph{$\diff(P)$} be the aggregate difference between the non-root vertices of $P$ and their respective parents. If $s_P$ denotes the shifted parent map of $P$, then 
\[
\diff(P) = \sum_{i} (i - s_P(i) + 1),
\]
where $i$ ranges over the non-root vertices of $P$.
These statistics admit a natural interpretation in terms of priority search on ordered forests.   
Given an ordered forest $F$ with priority traversal $\tau$, we say that a non-root vertex $i$ is a \emph{priority small ascent} if $i$ is visited immediately after its parent during the priority search of $F$. This condition is equivalent to requiring that vertex $\tau^{-1}(i)$ be a small ascent of the priority forest of $F$. 
The number of priority small ascents of $F$ is denoted by \emph{$\psa(F)$}.
The \emph{waiting time} of a non-root vertex $i$ of $F$ is defined as the number of steps during which $i$ is unblocked but not yet visited by priority search. If $i$ is a root, we set its waiting time to $0$. It follows that a vertex is a priority small ascent if and only if its waiting time is $1$. We denote by \emph{$\wait(F)$} the cumulative sum of waiting times over all vertices in $F$. Note that $\wait(F) = \diff(P)$, where $P$ is the priority forest of $F$.

 Let $\pi$ be an $(m,n)$-parking function. A car is said to be \emph{lucky} for $\pi$ if it parks in its preferred spot. The number of parking attempts (successful or not) before all cars find their parking spots is denoted by \emph{$\probes(\pi)$}.
Note that the displacement statistic of \cite{Stanley_parking}, defined as the total number of failed attempts before all cars find their parking spaces, is  a shift of the probes statistic.  
\begin{lem}
\label{lem: lucky and probes}
    Let $\pi$ be a partial parking function with priority forest $P$. Then, 
    \begin{enumerate}[label=(\alph*)]
        \item $\lucky(\pi)=\sasc(P)$,
        \item   
        \(
        \probes(\pi) = \diff(P).
        \)
    \end{enumerate}
\end{lem}
\begin{proof}
    Let $\omega$ denote the bird's eye permutation of $\pi$. By definition, 
    the contribution of car $i$ to the total number of linear probes  is $\omega^{-1}(i) -\pi(i) + 1$. Also, the shifted parent map of $P$ satisfies 
    $s_P(i) = \pi \circ \omega.$ 
    Hence, $s_P(\omega^{-1}(i)) = \pi(i)$ and the number of parking attempts of car $i$ is 
    $
    \omega^{-1}(i) - \pi(i) + 1 = \omega^{-1}(i) - s_P(\omega^{-1}(i)) + 1. 
    $
    Both claims follow from this identity.
\end{proof}

A  \emph{partial permutation record} is defined as a record (left-to-right maxima) of one of its maximal connected subwords.  
For instance, $2 - 143$ has records 1, 2 and 4. A \emph{partial parking function record} is a record of its partial bird's eye permutation.
A \emph{forest record} of $F$ is a nonroot node such that its label is the largest along the unique path from a root to that vertex. 
Tree record statistics have been studied enumeratively in  \cite{LRT-bijective-visit, LRT-GF-records,}. 
\begin{lem}
\label{lem: records from traversal}
Let $F$ be an ordered forest with priority traversal $\tau$. Then, 
$\Rec(F) = \Rec(\tau).$
\end{lem}
To see that  this lemma holds, it suffices to apply the proof of  \cite{LRT-Weary} to each component tree individually.
Note that the forest record statistic of $F$ cannot be read from its priority forest, but rather it is determined by its priority traversal (inverse Jordan-Hölder permutation).

\begin{thm}
\label{thm:equidistribution}
    Fix $m < n$. The two triplets of statistics 
    \[
    (\Rec, \psa, \wait), \qquad (\Rec, \lucky, \probes) 
    \]
    are jointly equidistributed when the first is consider over the set of ordered $(m,n)$-forests and the second over the set of $(m, n)$-parking functions.
\end{thm}
\begin{proof}
We show that these statistics are preserved by the forest weary bijection.
The forest weary bijection sends an ordered forest $F$ to a parking function $\pi_F$ with the same priority forest as $F$, say $P$. By Lemma \ref{lem: lucky and probes} we have 
$
\psa(F) = \sasc(P) = \lucky(\pi_F),$ and  $\wait(F) = \diff(P) = \probes(\pi_F).
$
On the other hand, the priority traversal of $F$ coincides with the bird's eye permutation of $\pi$, say $\omega$. Hence $\Rec(\pi_F) = \Rec(\omega)$ by definition, and 
 $\Rec(F) = \Rec(\omega)$ by Lemma \ref{lem: records from traversal}.
\end{proof}

We conclude this section by noting that some additional structural properties of ordered forests and partial parking functions are preserved by the weary bijection. For instance, the number of vertices in $F$ with exactly $k$ children coincides with the number of symbols in $[n]_0$ that appear in $\pi_F$ with multiplicity $k$ for each $k \geq 0$. This claim follows directly from \eqref{def:preference}.
Furthermore, any additional statistics defined solely in terms of the priority forest and the Jordan-Hölder permutation will be equidistributed across both families, and can be incorporated into the tuples in Theorem \ref{thm:equidistribution}. For example, for any ordered forest $F$, the number of inversions of the priority traversal of $F$ equals the number of inversions of the bird's eye permutation of $\pi_F$. The case of forest records is particularly interesting, as its usual definition does not involve the priority traversal.

\subsection{Parking functions and the priority lattice}

We close this section with a description of the priority lattice in the language of parking functions.  Assume that $P$ and $P'$ stand for two priority trees in $\Pi(n)$, and that $\rank P' = m'$ and $\rank P = m$.
The nodes of $\Pi(n)$ represent the priority forests of partial parking functions, with the exception of the top element $\hat{1}=\texttt{FAIL}$, which represents the class of all failed parking preferences. In contrast, the bottom element corresponds to the parking process in which no car has yet attempted to park.

 Furthermore, $P' \le P$ if and only if there exist partial parking functions $\pi'$ and $\pi$, whose priority forests are, respectively,  $P'$ and $ P$, and such that $\pi'$ is a restriction of $\pi$ to $[m']$.   If in addition $P' \lessdot P$, then Car $m$ parks at the spot designated by the edge labeling $\lambda(P', P)$. 
  The rank of $P$ is the number of cars that park under any partial parking function $\pi$ with priority forest $P$, whereas the corank of $P$ minus one is equal to the  number of empty spaces the bird's eye permutation $\omega_\pi$ of any such $\pi$.  Atoms are in bijection with parking spots, and coatoms with priority trees of parking functions.

 Theorem \ref{thm:maximal_chains_parking_functions} shows that  maximal chains of the proper interval $I_T = [\hat 0, T]$ of an increasing tree $T$ are in bijection with parking functions with priority tree $T.$ Is in this sense that our lattice refines the lattice of noncrossing partitions, see Stanley's \cite{StanleyMaximalChainsParking}. Moreover, given any priority forest $P$, the maximal chains in $I_P = [\hat 0, P] $ are in bijection with  $(m,n)$--parking functions with priority forest $P.$


\section{Properties of the priority lattice.}

\subsection{Enumerative results}

Let $\Pi(n,k)$ be the set of elements of rank $k$ in $\Pi(n)$.
The \emph{Whitney number of the second kind}, denoted by  $W_k$,  is defined as the number  of elements of rank $k$ of $\Pi(n).$ 
Fix $n$. Then, $W_k =|\Pi(n,k)|$ is different from zero if and only if $0 \leq k \leq n+1$. 

In general, 
$
(W_0,  W_1, \ldots    W_n)
$ is not a log-concave sequence. For $n=6$, we obtain
$
 1,$ $ 6,$ $ 20,$ $ 52,$ $ 126,$ $ 312,$ $ 720,$
and the triple  $(52, 126, 312)$ verifies that  $126^2 = 15876 <  16224  = 52\cdot312 $. This is the smallest counterexample to log concavity. However,  it is always an increasing sequence. 
\begin{table}[h!]
\centering
\renewcommand{\arraystretch}{1.3}
\begin{subtable}[t]{0.44\textwidth}
\centering
\small
\begin{tabular}{c|ccccccc}
    \toprule
    $n \backslash k$ & 1 & 2 & 3 & 4 & 5 & 6 & 7 \\
    \midrule
    0 & 1 & -1 & 0 & 0 & 0 & 0 & 0 \\
    1 & 1 & -1 & 0 & 0 & 0 & 0 & 0 \\
    2 & 1 & -2 & 1 & 0 & 0 & 0 & 0 \\
    3 & 1 & -3 & 3 & -1 & 0 & 0 & 0 \\
    4 & 1 & -4 & 6 & -4 & 1 & 0 & 0 \\
    5 & 1 & -5 & 10 & -10 & 5 & -1 & 0 \\
    6 & 1 & -6 & 15 & -20 & 15 & -6 & 1 \\
    \bottomrule
\end{tabular}
\caption{Whitney numbers of the first kind,
signed binomial coefficients.}
\label{tab:whitneyfirst}
\end{subtable}%
\qquad
\begin{subtable}[t]{0.44\textwidth}
\centering
\small
\begin{tabular}{c|cccccccc}
    \toprule
    $n \backslash k$ & 1 & 2 & 3 & 4 & 5 & 6 & 7 & 8 \\
    \midrule
    0 & 1 & 0 & 0 & 0 & 0 & 0 & 0 & 0 \\
    1 & 1 & 1 & 0 & 0 & 0 & 0 & 0 & 0 \\
    2 & 1 & 2 & 2 & 1 & 0 & 0 & 0 & 0 \\
    3 & 1 & 3 & 5 & 6 & 1 & 0 & 0 & 0 \\
    4 & 1 & 4 & 9 & 16 & 24 & 1 & 0 & 0 \\
    5 & 1 & 5 & 14 & 31 & 64 & 120 & 1 & 0 \\
    6 & 1 & 6 & 20 & 52 & 126 & 312 & 720 & 1 \\
    \bottomrule
\end{tabular}
\caption{Whitney numbers of the second kind, \url{https://oeis.org/A084938}.}
\label{tab:whitney_second}
\end{subtable}
\caption{}
\end{table}

\begin{lem}
\label{Whitney_increasing}
The sequence defined by the  Whitney numbers of the second kind of $\Pi(n)$ is increasing. Moreover, for $n\ge 3$, it is strictly increasing.
\end{lem}
\begin{proof} For $n\le2$ the result  is immediate and can be verified by inspection.
Therefore, we assume that $n\ge 3$, and fix  $k$ in $[n]$. To show that $|\Pi(n, k-1)| <|\Pi(n, k)|$ we define an injective, but not surjective, map $i_k : \Pi(n, k-1) \to \Pi(n, k)$.

Let $P$ be a priority forest of rank $k$ in $[2,n]$. Then, $P= T_0 T_1 \ldots T_{n-k}$.  Define $i_k(P)$ as the forest obtained by grafting the root of $T_1$ into the root of $T_0.$ The map $i_k$  has a left inverse. The preimage of a forest in the image is the forest resulting from the deletion of  the edge that joins  0 to its largest child. Therefore, $i_k$ is an injective map.

To see that $i_k$ is not a surjection, note that if $k \ne n$, then any increasing forest $T_0T_1$, where $T_0=\circ$ and $T_1$ is not trivial, has no preimage under $i_k$.  On the other hand, if $k=n$, increasing trees containing edge $(n-2,n)$ have no preimage under $i_k$. 
\end{proof}

  The \emph{corank generating function} of the priority lattices is defined as      \[
     \Co(z,t) 
     =
     \sum_{n\ge0} \sum_{k=0}^n  \ \big|\Pi(n, k)\big| \  z^n t^k.
    \]  
    The coefficient of $z^nt^k$ in $\Co(z,t)$  is the number of elements of corank $k$ in $\Pi(n).$

    \begin{lem}
    \label{le:corankgf}

  The corank generating function of the priority lattice $\Pi(n)$  is equal to the coefficient of $z^{n+1}$ in
       \[
     \Co(z,t)  = \frac{1}{1-t P(z)} + \frac{z}{1-z}
    \]
    where $
    P(z) = \sum_{n\ge1} (n-1)! \, z^n
    $ is the ordinary generating function for increasing trees according to the number of nodes. 
       \end{lem}

       \begin{proof}  The first summand comes from the bijection between increasing trees and permutations, the decomposition of a priority forest  as a sequence of rooted increasing trees, and the observation that  the number of component trees  in a forest of order $n$ with $k$ edges is $n-k.$ The second summand counts the top elements on each of the $\Pi(n)$'s.
       \end{proof}

The  generating function for sequence giving the cardinality of  $\Pi(n)\setminus \hat 1$ (the sequence obtained after ignoring the top element of $\Pi(n)$) is equal to $\frac{1}{1-t P(z)}$. Therefore, it is a shift of sequence \url{https://oeis.org/A051295} of the OEIS ``the number of permutations of $[n]$ such that the elements of each cycle of the permutation form an interval".

    \begin{figure}[h!]
    \centering
    \begin{subfigure}[t]{0.32\textwidth}
        \centering
        \includegraphics[width=\textwidth]{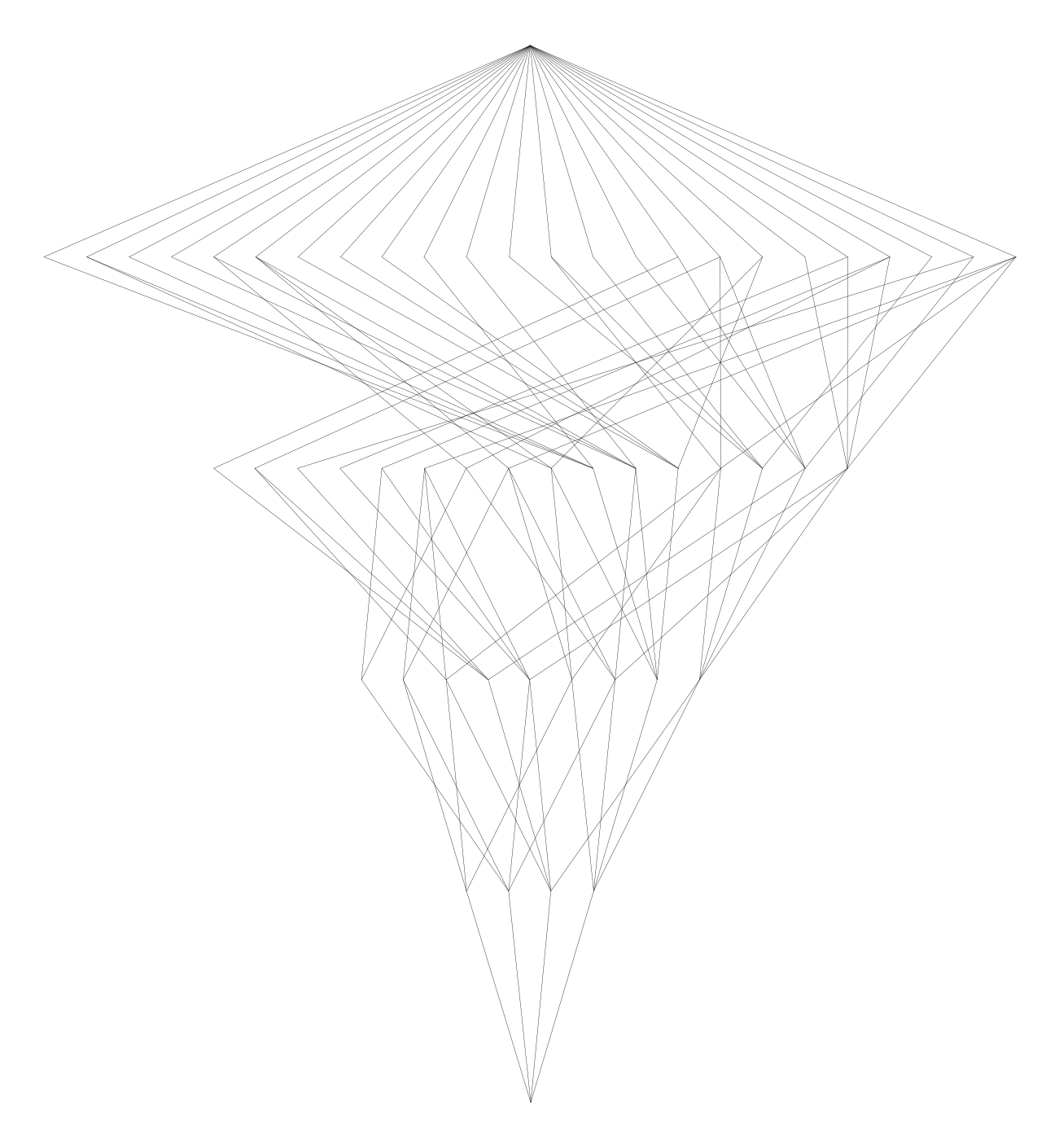}
        \caption{$n=4$}
    \end{subfigure}\hfill
    \begin{subfigure}[t]{0.32\textwidth}
        \centering
        \includegraphics[width=\textwidth]{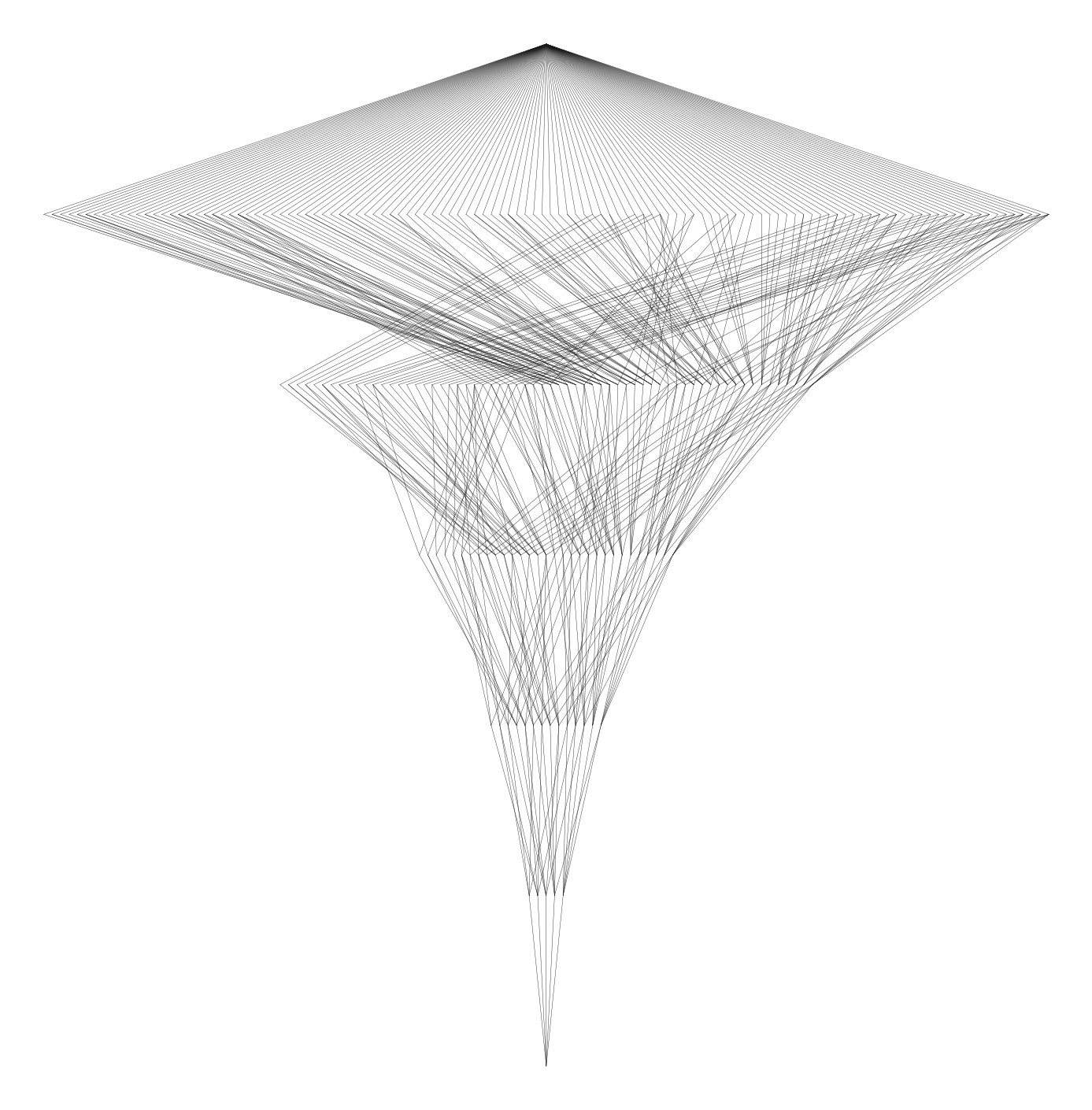}
        \caption{$n=5$}
    \end{subfigure}\hfill
    \begin{subfigure}[t]{0.32\textwidth}
        \centering
        \includegraphics[width=\textwidth]{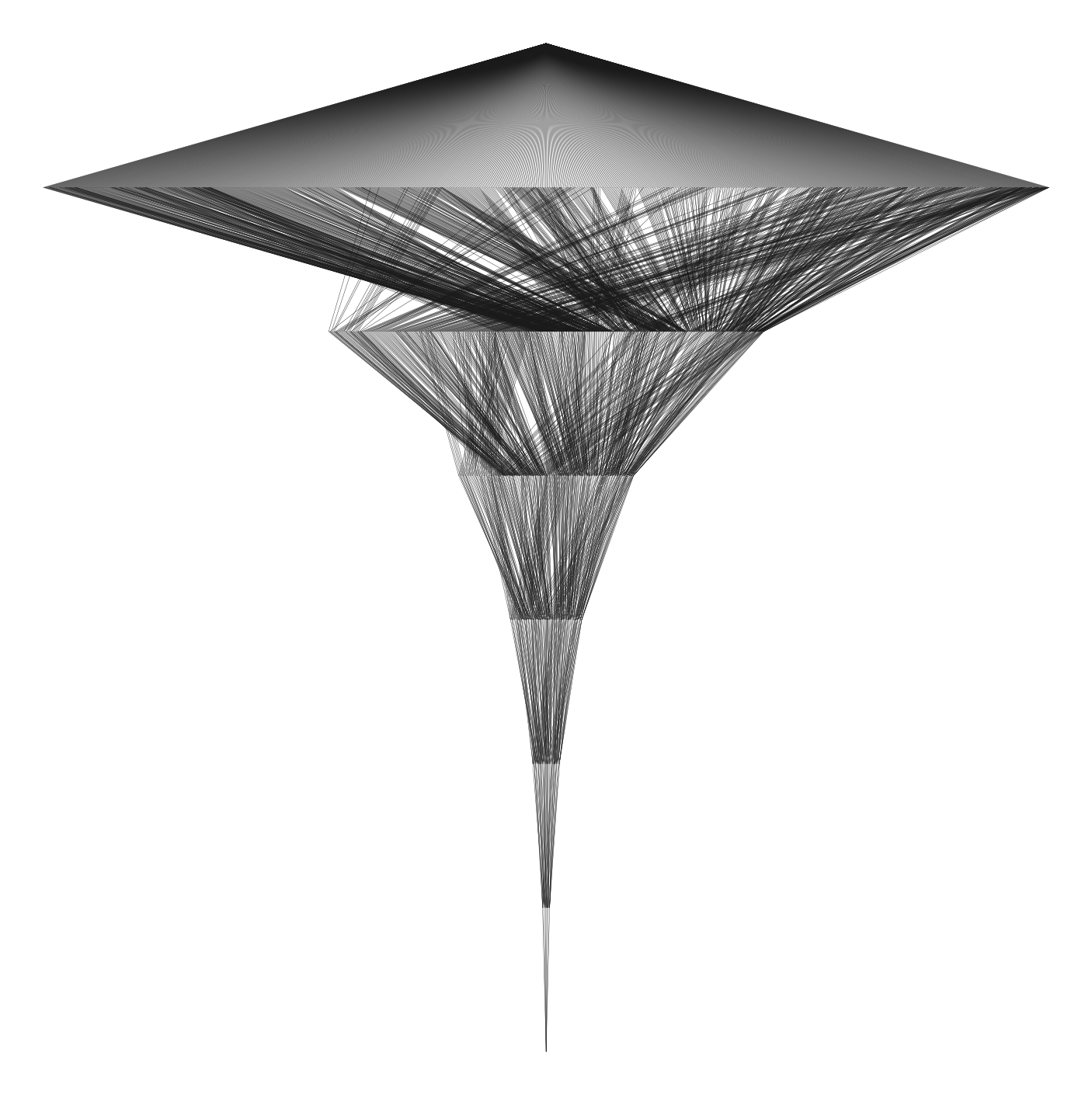}
        \caption{$n=6$}
    \end{subfigure}

    \caption{Priority lattices for increasing values of $n$.}
    \label{fig:priority-lattices}
\end{figure}

Since $\Pi(n)$ contains $n$ rank 1 but $n!$ corank 1 elements, the following corollary holds.

\begin{cor}
 The priority lattice $\Pi(n)$ is not self dual for $n\ge3$.
\end{cor}

On the other hand, for $n=0,1,2$, it can be easily checked that  $\Pi(n)$ is self dual.

\subsection{Maximal Interval Chains via the Hook-Length Formula.}

Gessel and Seo's hook length formula \cite{Gessel-Seo} was used in \cite{LRT-Weary} to compute the number of rooted trees and parking functions with a fixed priority tree $T$. In the framework of the priority lattice, this is the number of maximal chains of the principal ideal $[\hat 0, T]$.  We extend this result to ordered forests and partial parking functions by computing the number of maximal chains in $[\hat 0, P]$ for a given priority forest $P$.

Given a priority forest $P$ with non-root vertices $V \subseteq [n]$, we define a partial order $<_P$ on $V$ via the following cover relation: for $i, j \in V$ with $i < j$, we set $j \lessdot_P i$ if $j$ is the minimal label greater than $i$  such that $s_P(j) \leq i$. It is straightforward to verify that the Hasse diagram of the poset $(V, <_P)$, denoted by $H_P$, is a decreasing forest (see Figure \ref{fig:hooklength_poset}). 
For any $v$ in $V$, the \emph{hook-length} $h(v)$ is defined as the number of descendants of $v$ in $H_P$, including $v$ itself. Compare with Section 6.2 of \cite{LRT-Weary}, where the order $<_P$ is introduced in terms of parking function blueprints (for the case $m = n$).
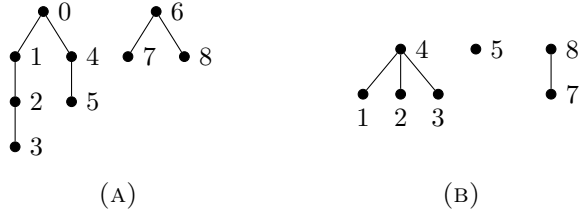
\begin{figure}[t]
    \centering
    \normalfont
    \begin{subfigure}[c]{0.2\textwidth}
        \centering
        \normalfont \small
\begin{tikzpicture}[
    level distance=0.6cm,
    level 1/.style={sibling distance=0.75cm},
    level 2/.style={sibling distance=0.75cm},
    record/.style={
        circle,
        draw=black,
        fill=black,
        inner sep=1.3pt,
        label={#1},
        solid
    },
]
\def\rootsep{1.5}
    \node[record, label={right:0}] (0) at (0*\rootsep, 0) {}
        child {node[record, label={right:1}] (1) {}
            child {node[record, label={right:2}] (2) {}
                child {node[record, label={right:3}] (3) {}}}}
        child {node[record, label={right:4}] (4) {}
        child {node[record, label={right:5}] (5) {}
        }};
    \node[record, label={right:6}] (6) at (1*\rootsep, 0) {}
        child {node[record, label={right:7}] (7) {}}
        child {node[record, label={right:8}] (8) {}};
\end{tikzpicture}
        \caption{}
    \end{subfigure}%
    \qquad \qquad
    \begin{subfigure}[c]{0.2\textwidth}%
        \centering
            \normalfont \small
\begin{tikzpicture}[
    level distance=0.6cm,
    level 1/.style={sibling distance=0.5cm},
    level 2/.style={sibling distance=0.75cm},
    record/.style={
        circle,
        draw=black,
        fill=black,
        inner sep=1.3pt,
        label={#1},
        solid
    },
]
\def\rootsep{1}
    \node[record, label={right:4}] (4) at (0*\rootsep, 0) {}
        child {node[record, label={below:1}] (1) {}}
        child {node[record, label={below:2}] (2) {}}
        child {node[record, label={below:3}] (3) {}};
    \node[record, label={right:5}] (5) at (1*\rootsep, 0) {};
    \node[record, label={right:8}] (8) at (2*\rootsep, 0) {}
        child {node[record, label={right:7}] (7) {}};
    \phantom{
    \node[record, label={right:0}] (0) at (1*\rootsep, 0.5cm) {}
        child {node[record, label={right:1}] (1) {}
            child {node[record, label={right:2}] (2) {}
                child {node[record, label={right:3}] (3) {}}}}
        child {node[record, label={right:4}] (4) {}};
    }
\end{tikzpicture}%
        \caption{}
    \end{subfigure}
    \caption{(a) A priority $(7, 8)$-forest $P$.
    (b) The Hasse diagram $H_P$. }
    \label{fig:hooklength_poset}
\end{figure}
\begin{p}
    Fix a priority $(m, n)$-forest $P$ with non-root vertices $V \subseteq [n]$. Then, the number of maximal chains in $[\hat 0, P]$ is given by 
    \[
    \frac{m!}{\prod_{v\in V} h(v)}.
    \]
    This formula also counts the number of ordered $(m, n)$-forests and the number of  $(m, n)$-parking functions with priority forest $P$.
\end{p}
\begin{proof}
    Write $P = (T_0, T_1, \dots, T_{n-m})$. Let $V_i$ be the set of non-root vertices of $T_i$ and denote $|V_i| =t_i$, so that $t_0 + t_1+ \cdots + t_{n-m} = m$.  
    We build an ordered $(m,n)$-forest $F$ with priority forest $P$ in two steps.
    First, we partition the $m$ non-root labels among the trees of $F$, which can be done in $\binom{m}{t_0, t_1, \cdots, t_{n-m}}$ ways. Then, for each $i$, we construct a rooted tree whose priority tree is $T_i$ (up to a constant shift of the labels). By \cite[Lemma 6.3]{LRT-Weary}, this can be done in ${t_i!}{\prod_{v \in V_i} h(v)^{-1}}$ ways. It follows that the number of ordered $(m, n)$-forests with priority forest $P$ is 
    \[ 
    \binom{m}{t_0, t_1, \dots, t_{n-m}}\frac{t_0! t_1! \cdots t_{n-m}!}{\prod_{v \in V} h(v)} = \frac{m!}{\prod_{v\in V }h(v)},
    \]
    as desired. 
\end{proof}
In particular, when $P$ is a priority tree, we recover Lemma 6.3 of \cite{LRT-Weary}.

\subsection{Topology of the priority lattice.} In this section, we investigate the topology of the priority lattice. We show that its proper intervals are distributive, whereas the distinctive nature of the top element substantially complicates the study of intervals containing $\hat{1}$.

 \begin{lem}
 The lattice $\Pi(n)$ is not  distributive for $n \ge 3$.
 \end{lem}
 \begin{proof}
 Figure \ref{Fig:diamond_lattice} describes a copy of the diamond lattice that sits inside $\Pi(n)$ when  $n \ge 3$.
 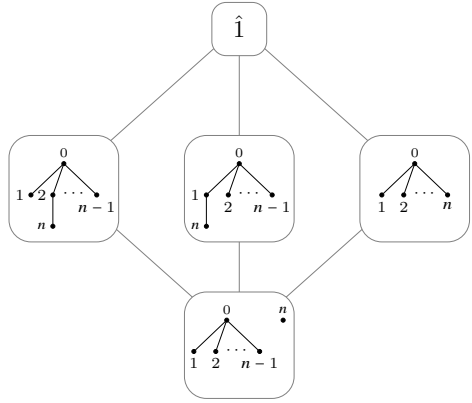
\begin{figure}[h!]
    \centering
    \resizebox{0.42\textwidth}{!}{\tikzset{
    treenode/.style={
        circle,
        draw=black,
        fill=black,
        inner sep=0.65pt,
        outer sep=0pt,
        label={#1},
        solid
    },
    every label/.style={
            font=\tiny\fontsize{6}{7}\selectfont,
            label distance = 0.1pt,
            inner sep=1pt,
            outer sep=1pt,
        }
}
\begin{tikzpicture}[
    scale=1,
    level distance=0.5cm,
    level 1/.style={sibling distance=0.35cm},
    level 2/.style={sibling distance=0.35cm},
    pics/one/.style = {
        code = {
            \coordinate (#1-root) at (0,0);
            \node (one) at (0, 0) { $\hat 1$};
            };
    },
pics/forest0/.style = {
    code = {
        \def\rootsep{0.45cm}
        \def\yoffset{0.3cm}
        \def\xoffset{-0.2cm}
        \coordinate (#1-root) at (0,0);
        \node[treenode,label={above:0}] (F2-T0-V-0) at (\xoffset,\yoffset) {}
            child { node[treenode,label={below:\tiny 1}] (F2-T0-V-1) {}}
        child { node[treenode,label={below:\tiny $2$}] (F2-T0-V-2) {} }
            child [missing]
            child { node[treenode,label={below:\tiny$n-1$}] (F2-T0-V-3) {}
            };
        \node (dots) at (0.175cm +\xoffset, -0.66*\yoffset) {\tiny $\cdots$};
        \node[treenode, label={above:$n$}] (nroot) at (2*\rootsep + \xoffset, \yoffset) {};
    }
},
pics/forest1/.style = {
    code = {
        \def\rootsep{0.45}
        \def\yoffset{0.3cm}
        \coordinate (#1-root) at (0,0);
        \node[treenode,label={above:0}] (F2-T0-V-0) at (0,\yoffset) {}
            child { node[treenode,label={below:\tiny 1}] (F2-T0-V-1) {}}
            child { node[treenode,label={below:\tiny $2$}] (F2-T0-V-2) {}}
            child [missing]
            child { node[treenode,label={below:\tiny$n$}] (F2-T0-V-3) {}};
        \node (dots) at (0.175cm, -0.66*\yoffset) {\tiny $\cdots$};
    }
},
pics/forest2/.style = {
    code = {
        \def\rootsep{0.45}
        \def\yoffset{0.3cm}
        \coordinate (#1-root) at (0,0);
        \node[treenode,label={above:0}] (F2-T0-V-0) at (0,\yoffset) {}
            child { node[treenode,label={left:\tiny 1}] (F2-T0-V-1) {} 
            child { node[treenode,label={left:$n$}] (F2-T0-V-2) {} }}
        child { node[treenode,label={below:\tiny $2$}] (F2-T0-V-2) {} }
            child [missing]
            child { node[treenode,label={below:\tiny$n-1$}] (F2-T0-V-3) {}
            };
        \node (dots) at (0.175cm, -0.66*\yoffset) {\tiny $\cdots$};
    }
},
pics/forest3/.style = {
    code = {
        \def\rootsep{0.45}
        \def\yoffset{0.3cm}
        \coordinate (#1-root) at (0,0);
        \node[treenode,label={above:0}] (F2-T0-V-0) at (0,\yoffset) {}
            child { node[treenode,label={left:\tiny $1$}] (F2-T0-V-2) {} }
            child { node[treenode,label={left:\tiny 2}] (F2-T0-V-1) {} 
            child { node[treenode,label={left:$n$}] (F2-T0-V-2) {} }}
            child [missing]
            child { node[treenode,label={below:\tiny$n-1$}] (F2-T0-V-3) {}
            };
        \node (dots) at (0.175cm, -0.66*\yoffset) {\tiny $\cdots$};
    }
},
    ]

\def\rowysep{2.5}
\def\ycoord{0}
\def\xsep{3.5}

\pic (f0) at (0,0) {forest0};

\def\ycoord{\rowysep}

\def\xsep{3.5}
\pic (f3) at (-0.8*\xsep,\ycoord) {forest3};
\pic (f2) at (0, \ycoord) {forest2};
\pic (f1) at (0.8*\xsep, \ycoord) {forest1};


\def\ycoord{2*\rowysep}
\pic (one) at (0, \ycoord) {one};

\begin{pgfonlayer}{background}

\draw[gray] (f0-root) --  (f1-root);
\draw[gray] (f0-root) --  (f2-root);
\draw[gray] (f0-root) --  (f3-root);
\foreach \f in {f1,f2,f3}{
    \draw[gray] (\f-root) -- (one-root);
}
\def\circr{12mm}

\def\boxw{1.75cm} 
\def\boxh{1.5cm} 
\def\cornerr{10pt} 

\foreach \f in {f0,f1,f2,f3}{
  \draw[gray,rounded corners=\cornerr, fill=white]
    ($(\f-root)+(-0.5*\boxw,-0.5*\boxh) - (0mm, 2mm)$) rectangle
    ($(\f-root)+(0.5*\boxw,0.5*\boxh)$);
}
\draw[gray, rounded corners =0.65*\cornerr, fill=white] 
($(one-root)+(-0.25*\boxw,-0.25*\boxh) - (0mm, 1mm)$) rectangle
    ($(one-root)+(0.25*\boxw,0.25*\boxh)$);

\end{pgfonlayer}

\end{tikzpicture}}
    \caption{  For each $n\ge 3$, 
   $\Pi(n)$ contains a copy of the diamond lattice. }
\label{Fig:diamond_lattice}
\end{figure}
  \end{proof}
  
For $n=0, 1, 2$, the lattices $\Pi(n)$ are distributive, as can easily be checked by inspection.
On the other hand, all intervals $[P', P]$ (with $P' < P$  priority forests) are  distributive.

\begin{lem} 
\label{intervales_are_distributive} Let $P' < P$ be priority forests. Then, the interval   $[P', P]$ together with the meet and joint operations, is a distributive lattice. 

\end{lem}

\begin{proof}
When we restrict to a proper ideal, the meet and join operations restrict to the union and intersection of the respective edge sets. On the other hand, the  union and intersection satisfy the distributive laws.
\end{proof}

We now compute the Möbius functions of the proper intervals of $\Pi(n)$. To this end, we recall some definitions.
Let $C$  be $s = s_0 < s_1 < \dots < s_k = t$ be a maximal chain of an interval $[s, t]$ of a poset $P$, let $\lambda$ be an edge labeling for $P$. Set
\[
\lambda(C) = (\lambda(s_0, s_1), \lambda(s_1, s_2), \dots, \lambda(s_{k-1}, s_k)).
\]

The chain $C$ is \emph{increasing} if
$
\lambda(s_0, s_1) \le \lambda(s_1, s_2) \le \dots \le \lambda(s_{k-1}, s_k)
$, and \emph{ascent-free} if
$
\lambda(s_0, s_1) \not< \lambda(s_1, s_2) \not< \dots \not< \lambda(s_{k-1}, s_k).
$
We say that the edge labeling is an \emph{ER-labeling}   if every interval $[s, t]$ has exactly one increasing maximal chain $m$.
Moreover, we say that an ER-labeling is an \emph{EL-labeling} if
in each interval, the unique increasing maximal chain also precedes every other chain in lexicographic order. Finally, we define 
\[
\mathcal{AF}_{[x,y]}= \{C \ | \ C 
\text{ is an ascent-free maximal chain in } 
[x, y]\}.
\]

 \begin{lem} The $P$ be a priority forest. The edge labeling of $I_P=[\hat0,P]$ defined by the Jordan-Hölder permutations is an EL-labeling.
 \end{lem}

 \begin{proof} Let $L$ be the set of all non-root notes of $P$.
We construct a maximal chain  $C$ of $I_P$  starting from $P$ and successively deleting the edge joining the largest available vertex with its parent, until no edge remains.
By construction $\lambda(C)$ is the identity permutation of $\SS_L$, and $C$ is a increasing chain. 
Since the labeling of maximal chains is always a permutation of $\SS_L$,  there is a unique increasing chain of $I_P$, and the second condition in the definition of EL-labeling
follows  trivially.
\end{proof}

Let \(P < P'\) be priority forests, and let \(e \in P' \setminus P\). We say that \(e\) is \emph{\(P\)-removable} if removing \(e\) from \(P'\) yields a priority forest. Otherwise, \(e\) is called a \emph{\(P\)-ascent} of $P'$. In Lemma~\ref{lem:char_asc_free_general_interval}, we show that \(e\) is a \emph{\(P\)-ascent} if and only if it forces the presence of an ascent in the Jordan--Hölder word of every maximal chain in \([P, P']\). When $P=\hat0$, we simply refer to the edge as either removable or as an ascent, omitting the prefix.

 \begin{lem} 

  \label{lem:char_asc_free_general_interval}
Let $P<P'$ be priority forests. The presence of a \(P\)-ascent in \(P'\) forces an ascent in the Jordan--Hölder permutation of every maximal chain in \([P, P']\).
 \end{lem}
 \begin{proof}
An edge \(e\) is a \(P\)-ascent if removing $e$ from \(P'\) produces a forest whose trees’ labels are no longer integer intervals. In particular, this would imply the existence of nodes \(i < j\) with \(j \in T_k\) and \(i \in T_{k+1}\), where \(T_k\) and \(T_{k+1}\) are consecutive component trees of \(P' \setminus \{e\}\). To avoid this problem  \(i\) should attach to a parent before \(j\). Therefore, in any chain from \(P\) to \(P'\), the label \(i\) precedes \(j\), creating an ascent in the chain labeling.
 \end{proof}
 
  \begin{lem} 
  \label{lem:char_removable}
Let $P<P'$ be priority forests. If all edges of $P\setminus P'$ are removable, then there is exactly one maximal chain in $[P, P']$ whose  Jordan-Hölder permutation  is ascent-free. Moreover, $[P, P']$  is isomorphic to the Boolean lattice of order  $\rank(P')-\rank(P).$
 \end{lem}
 \begin{proof}  We describe the maximal chains from top to bottom. First, observe that under these conditions, edges can be removed from $P\setminus P'$ in any possible order. Upon tracking the children nodes, all permutations appear. On the other hand, the maximal chain whose Jordan-Hölder permutation is ascent-free is the one obtained after deleting the edge attaching node \(j\) its parent in \(P'\) following the order \(1, 2, \dots, n\). 
 \end{proof}

Let $\pi$ be a parking function with priority forest $P$, and let $e$ be an edge of $P$ with child node $x$. If $e$ is an ascent edge, then the bird’s-eye permutation of $\pi$ has an ascent between positions $x$ and $x+1$. Whereas if $e$ is a removable edge, then the positions $x$ and $x+1$ in the bird’s-eye permutation could be either an ascent or a descent.

After Lemma \ref{lem:char_removable}, we call a priority forest in which every edge is removable a \emph{Boolean forest}. Since siblings nodes imply the existence of ascents, any Boolean forest is necessarily a union of path graphs.  In the language of parking functions, any maximal chain in a Boolean forest corresponds bijectively to a partial parking function in which all cars are lucky.  
More generally, given priority forests \(P < P'\), we say that \(P'\) is a \emph{\(P\)-Boolean} if every edge in \(E(P') \setminus E(P)\) is $P$-removable. A $\hat0$-Boolean forest is a Boolean forest.

A powerful result by Richard Stanley allows us to compute the values of the Möbius function for all interval not containing $\hat1$, see \cite{Stanley_finite_lattice_jordal_holder,  Whitney_twins}.
The values of the Möbius functions  of the principal ideals of $\Pi(2)$ and $\Pi(3)$ appear, respectively, in  Figures \ref{fig:L2} and \ref{fig:Pi3}. 

 \begin{thm}[Stanley, \cite{Stanley_finite_lattice_jordal_holder}]
 \label{Thm:Stanley_mobius}
Let $P$ be a graded poset with an ER-labeling $\lambda$. 
Then for every $x < y$ in $P$,  we have that
$
\mu(x, y) = (-1)^{\rho(y) - \rho(x)} \,
\big|\vspace{1pt} \mathcal{AF}_{[x,y]}\vspace{1pt} |.
$
 \end{thm}

  \begin{lem} 
  \label{lem:Mobius_priority_forests}
  Let $P < P'$ be priority forests. Then,  \[
  \mu(P, P')= 
  \begin{cases}
  (-1)^{\rho(P')-\rho(P)} &\text{ if $P'$ is a $P$-Boolean forest.}\\
  0 &\text{ otherwise}
  \end{cases}
  \]

 \end{lem}

  \begin{cor} 
  \label{cor:Complete_Mobius}
  Fix $n\ge1$, then 
 $\mu(\hat0,\hat1)=0$. 
 \end{cor}
 \begin{proof}
 As the only priority forests with nonzero Möbius function are the Boolean forest, and the only corank one Boolean forest is the increasing path graph, and  since $\hat 1$ covers all priority forests, the sum of the Möbius values of all priority trees is equal to zero. Therefore, $\mu(\hat0,\hat1)=0$ 
 \end{proof}

It only remains to compute the value of the Möbius function on intervals that contain the top element. This final calculation is carried out using a sign-reversing involution.

\begin{lem}
\label{lem:mobius_upper_intervals}
    Let $P $ be a priority forest. 
        If $P$ is a tree, then $ \mu(P,\hat 1)=-1$. Otherwise, $P = T_0 \ T_1 \ldots T_\ell$ with $\ell\ge1$, and  $ \mu(P,\hat 1)$  is equal to 
    \[
    \mu(P,\hat 1)  = (-1)^{\corank P} \prod_{k=0}^{\ell-1} e(T_k) 
    \]
    where we write $e(T)$ for the number of edges of $T$.

\end{lem}

\begin{proof} If $P$ is a tree, the result is immediate. 
Otherwise, as $\hat1$ it the top element, we need to compute the sum of the values of  $\mu(P,P')$ for all  all priority trees $P'\ge P,$ and reverse the sign of the result.
Let $P = T_0 \ T_1 \ldots T_\ell$, $\ell\ge1$, and define \(e_k\) as the edge joining the roots of consecutive component trees \(T_{k-1}\) and \(T_k\).

Define an involution on the set of priority trees in the interval $[P,\hat{1}]$ as follows. Let $P' \in [P,\hat{1}]$.
If $e_1$ is an edge of $P'$, we remove $e_1$. If $e_1$ is not an edge of $P'$ and adding $e_1$ does not create a cycle, we add $e_1$. If neither case applies, we repeat the same procedure with $e_2$, and if necessary continue successively with $e_3, \ldots, e_\ell$ until one of these two operations can be carried out. If no such edge exists, then $P'$ is a fixed point of the involution.

Since non-removable edges induce ascents in its Jordan-Hölder permutations, forcing the Möbius function to vanish. The involution, on the other hand, affects only removable edges. Hence, it maps $P$-Boolean forests to $P$-Boolean forests and no $P$-Boolean forests to no $P$-Boolean forests. Finally, since forests in an orbit differ by exactly one edge, this reverses the parity and makes the involution sign-reversing.
Finally, observe that the fixed points of the involution are  trees, as otherwise there would be two consecutive roots that could potentially be attached at their roots. 

Let $T$ be a fixed point of the involution. 
If  $T$ is not a $P$-Boolean forest, we ignore it as $\mu(T,\hat1)=0$. On the other hand, if $T$ is a $P$-Boolean forests, then $\mu(P,T)=(-1)^{\corank P -1}$, implying that there are no further cancellations. To compute the value of the Möbius function \(\mu(P,\hat{1})\), count them and change the sign of the result.
There are a total of 
$
\prod_{k=0}^{\ell-1} \big(|T_k|-1\big)
$
of them,  as they can be constructed by attaching the root of \(T_k\) to any non-root node of \(T_{k-1}\) for each \(k = 1, 2, \ldots, \ell\).
\end{proof}

    \begin{figure}[h!]
    \centering
    \resizebox{0.8\textwidth}{!}{\input{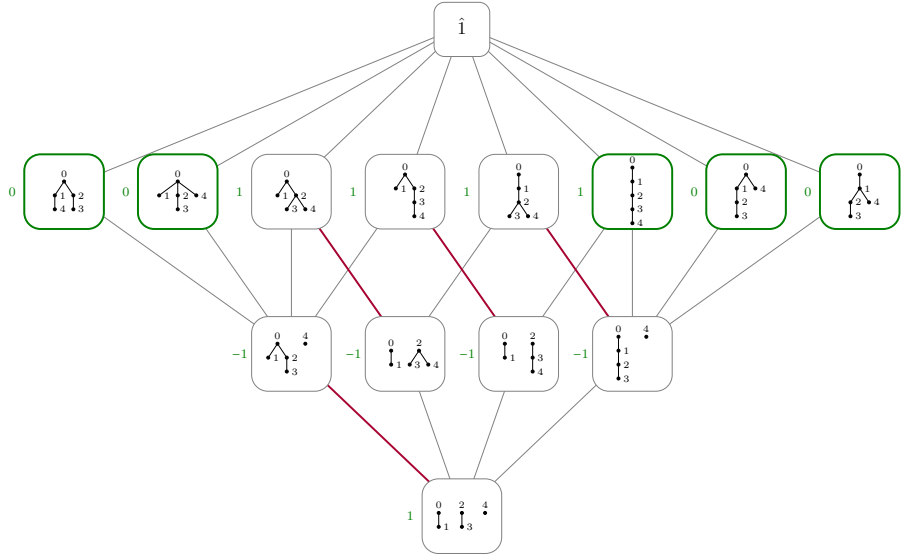}}
    \caption{An example of the sign-reversing involution with its  fixed points in green, and the two-orbits linked with crimson edges. }
\label{Fig:involution}

\end{figure}

\begin{ex}
Figure~\ref{Fig:involution} illustrates the sign-reversing involution on an upper interval of $\Pi(5)$. Fixed points are shaded, and $2$-orbits are indicated by dark lines. 
\end{ex}
Some immediate consequences of Lemma \ref{lem:mobius_upper_intervals}.
Assume that in all constructions that we will describe labels are chosen so as to produce a priority forest, and recall that trees in $\Pi(n)$ have $n$ edges.  The Möbius function  $\mu(P,\hat 1)=0$ if and only if  $F$ contains a component tree $T_k =\circ$ with $k\ne \ell$.
Moreover, for each $k \in [n-1]$, there exists a priority forest $P = T_0 \, T_1 $ of corank 2 (and thus, $n-1$ edges) with $\mu(P, \hat{1}) = k$. For example, if $T_0$ is a tree with  $k$ edges and $T_1$ has $n-k-1$ edges. 
More generally, for each $k \in [n-m]$, there exists a priority forest $P = T_0 \, T_1 \, T_2 \,\ldots  \,  T_m$ of corank $m + 1$ with $\mu(P, \hat{1}) = (-1)^{m+1} \ k$. For example, for any priority forest with three components $T_0$, $T_1$, and $T_2$ with, respectively, $k$, one, and $n-m$ edges. 
Finally, observe that the AM–GM inequality implies that the signless Möbius function of a priority forest $P = T_0 \, T_1 \, T_2 \,\ldots  \,  T_{n-m}$ of corank $n - m + 1$  satisfies that
\[
\left|  \hspace{2pt} \mu(P, \hat{1}) \hspace{1pt} \right| = 
e(T_0)e(T_1)\ldots e(T_{n-m}) \le \left(\frac{m}{n-m+1}\right)^{n-m+1}. 
\]

We conclude this section  with a brief remark on the topology of the order complex associated with intervals in the priority lattice.
Let $P<P'$ be priority forests.  If $P'$ is a $P$-Boolean forest, then the order complex $\Delta((P,P'))$  has the homotopy type of a sphere of dimension $\rank(P')-\rank(P)-2$. However, if $P'$ is not a $P$-Boolean forest, then  $\Delta((P,P'))$ has the homotopy type of a point. These observations are an immediate consequence of the following result by Anders Björner and Michelle Wachs \cite{Bjorner_Wachs_lex_shellable}.

 \begin{thm}[Björner and Wachs \cite{Bjorner_Wachs_lex_shellable}]
Let $P$ be a graded poset with an ER-labeling $\lambda$. For  $x < y$ fixed.
Then, the order complex $\Delta((x, y))$ is shellable. Moreover, it has the homotopy type of a
wedge of 
$| \vspace{1pt} \mathcal{AF}_{[x,y]}| \vspace{1pt}$
 spheres each of dimension $\rho([x, y]) - 2$. As a consequence, $[x, y]$ is Cohen-Macaulay.
Finally, $\mathcal{AF}_{[x,y]}$
forms a basis for the top reduced cohomology $\widetilde{H}^{\rho([x,y])-2}(\vspace{1pt} (x, y) \vspace{1pt})$ of $\Delta((x, y))$.

 \end{thm}

\smallskip


\subsection{The characteristic polynomial}
The \emph{characteristic polynomial} of 
$\Pi(n)$ is defined as
\[
\chi(\Pi(n),q) = \sum_{x \in P}
\mu(\hat 0, x) \
q^{\rho(\hat 1)-\rho(x)}
\]

Let $P$ be a priority forest, and let $\pi$ be a  partial parking function.  Recall that  \emph{$\sasc(P)$} denotes the number of small ascents of $P$ and \emph{$\lucky(\pi)$} the number of lucky cars (cars that park at their favorite spot) under the $\pi$. 
    
    \begin{thm}
    \label{thm_charpoly}
    Let $P$ be a $(m,n)$ priority forest, and let $I_P$ be its order ideal. Then,
    \begin{align*}
  \chi(I_P,q) =   q^{m-\sasc(P)} (q-1)^{\sasc(P)} = q^{m-\lucky(\pi)} (q-1)^{\lucky(\pi)}.
    \end{align*}

    \end{thm}

    \begin{proof}
 We recover the Boolean forest $W$ of $P$  by deleting from $P$ all edges that are not small ascents.  Therefore, $I_W$ is isomorphic to the Boolean algebra of rank $\sasc(P).$ See Figure \ref{Fig:maximal_wood}.
 The factor of
 $q^{n-\sasc(P)}$, 
 on the other hand, keeps track of the forest of corank smaller than the corank of  $W$, that are automatically zero.
    \end{proof}

\begin{figure}[h!]
    \centering
    \begin{subfigure}{0.2\textwidth}
        \centering
\pgfdeclarelayer{bg}
\pgfsetlayers{bg,main}

\makeatletter
\pgfkeys{
  /tikz/on layer/.code={
    \pgfonlayer{#1}\begingroup
    \aftergroup\endpgfonlayer
    \aftergroup\endgroup
  }
}
\def\node@on@layer{\aftergroup\node@@on@layer}
\makeatother

\tikzset{
    treenode/.style={
        circle,
        draw=black,
        fill=black,
        inner sep=0.65pt,
        outer sep=0pt,
        solid
    },
    every label/.style={
        font=\tiny\fontsize{6}{7}\selectfont,
        label distance = 0.1pt,
        inner sep=1pt,
        outer sep=1pt,
    }
}
\resizebox{\textwidth}{!}{
\begin{tikzpicture}[
    level distance=0.45cm,
    level 1/.style={sibling distance=1cm},
    level 2/.style={sibling distance=0.6cm},
    edge from parent/.style={draw,on layer=bg}
]
\node[treenode,label={right:0}] {}
child { node[treenode,label={right:1}] {}
    child { node[treenode,label={right:6}] {}
        child { node[treenode,label={right:7}] {} }
        child { node[treenode,label={right:8}] {} }}}
child { node[treenode,label={right:2}] {}
    child { node[treenode,label={right:3}] {}
        child { node[treenode,label={right:4}] {} }}
    child { node[treenode,label={right:5}] {} }};
\end{tikzpicture}
}
        \caption{}
    \end{subfigure}
    \qquad
    \qquad
    \begin{subfigure}{0.2\textwidth}
        \centering
\pgfdeclarelayer{bg}
\pgfsetlayers{bg,main}

\makeatletter
\pgfkeys{
  /tikz/on layer/.code={
    \pgfonlayer{#1}\begingroup
    \aftergroup\endpgfonlayer
    \aftergroup\endgroup
  }
}
\def\node@on@layer{\aftergroup\node@@on@layer}
\makeatother

\tikzset{
    treenode/.style={
        circle,
        draw=black,
        fill=black,
        inner sep=0.65pt,
        outer sep=0pt,
        solid
    },
    every label/.style={
        font=\tiny\fontsize{6}{7}\selectfont,
        label distance = 0.1pt,
        inner sep=1pt,
        outer sep=1pt,
    }
}

\resizebox{\textwidth}{!}{
\begin{tikzpicture}[
    level distance=0.45cm,
    level 1/.style={sibling distance=1cm},
    level 2/.style={sibling distance=0.6cm},
    edge from parent/.style={draw,on layer=bg}
]
\node[treenode,label={right:0}] {}
child[edge from parent/.append style={draw=USred,semithick}]
{ node[treenode,label={right:1}, USred] {}
    child[edge from parent/.append style={draw=black, thin, densely dotted}]
    { node[treenode,label={right:6}] {}
        child[edge from parent/.append style={draw=USred, semithick, solid}]
        { node[treenode,label={right:7},USred] {} }
        child { node[treenode,label={right:8}] {} }}}
child[edge from parent/.append style={draw=black, thin, densely dotted}]
{ node[treenode,label={right:2}] {}
    child[edge from parent/.append style={draw=USred,semithick,solid}]
    { node[treenode,label={right:3},USred] {}
        child[edge from parent/.append style={draw=USred,semithick,solid}]
        { node[treenode,label={right:4}, USred] {} }}
    child { node[treenode,label={right:5}] {} }};
\end{tikzpicture}}
        \caption{}
    \end{subfigure}
    \qquad
    \qquad
    \begin{subfigure}{0.2\textwidth}
        \centering
\pgfdeclarelayer{bg}
\pgfsetlayers{bg,main}

\makeatletter
\pgfkeys{
  /tikz/on layer/.code={
    \pgfonlayer{#1}\begingroup
    \aftergroup\endpgfonlayer
    \aftergroup\endgroup
  }
}
\def\node@on@layer{\aftergroup\node@@on@layer}
\makeatother

\tikzset{
    treenode/.style={
        circle,
        draw=black,
        fill=black,
        inner sep=0.65pt,
        outer sep=0pt,
        solid
    },
    every label/.style={
        font=\tiny\fontsize{6}{7}\selectfont,
        label distance = 0.1pt,
        inner sep=1pt,
        outer sep=1pt,
    }
}

\resizebox{\textwidth}{!}{
\begin{tikzpicture}[
    level distance=0.5cm,
    level 1/.style={sibling distance=1cm},
    level 2/.style={sibling distance=0.6cm},
    edge from parent/.style={draw,on layer=bg}
]
\def\rootsep{0.5cm}
\node[treenode, label={above:0}] (F10-T0-V-0) at (0*\rootsep,0)  {}
child { node[treenode, label={right:1}] (F10-T0-V-1) {} };

\node[treenode, label={above:2}] (F10-T1-V-2) at (1*\rootsep,0) {}
child { node[treenode, label={right:3}] (F10-T1-V-3) {}
    child { node[treenode, label={right:4}] (F10-T1-V-4) {} } };

\node[treenode, label={above:5}] (F10-T2-V-5) at (2*\rootsep,0) {};
\node[treenode, label={above:6}] (F10-T3-V-6) at (3*\rootsep,0) {}
child { node[treenode, label={right:7}] (F10-T3-V-7) {} };
\node[treenode, label={above:8}] (F10-T4-V-8) at (4*\rootsep,0) {};
\end{tikzpicture}}
        \caption{}
    \end{subfigure}
    \caption{The Boolean forest of maximal rank in a priority tree.}
    \label{Fig:maximal_wood}
\end{figure}
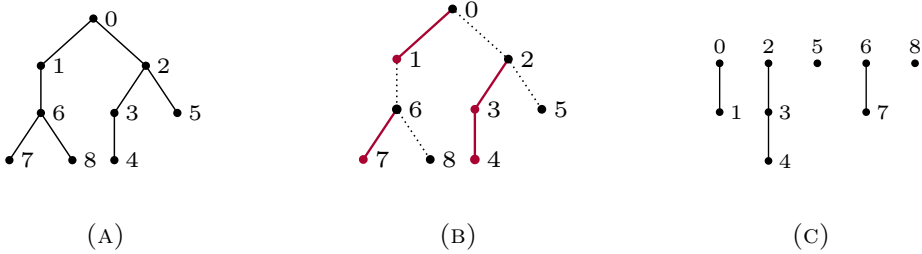  

    \begin{cor} For $n\ge1$, the characteristic polynomial of $\Pi(n)$ is 
    $
    q (q-1)^{n}
    $
        \end{cor}
    \begin{proof} The only forests $P$  of $\Pi(n)$ with Möbius function $\mu(\hat0,P)\neq0$  are the Boolean forests, and when we restrict to $P$-Boolean forests we get a copy of the boolean algebra $B_n$. On the other hand, $\mu(\hat0,\hat1)=0$, hence we get a factor of $q.$
 \end{proof}

The \emph{Whitney numbers of 1st kind}, denoted by $w_k$, are defined as the sum of the values of the Möbius function of all rank $k$ elements of a poset. They are the coefficients of their respective characteristic polynomials. Therefore,  for forest ideals, they are equal to
\[
w_k(I_P) = (-1)^{k} \binom{\sasc(P)}{k}
\]
On the other hand, the Whitney of 1st kind $w_k$ of $\Pi(n)$ are given by  
$w_k(\Pi(n)) =(-1)^k \binom{n}{k}$ for $k \in [n-1]$, and 0 otherwise.
\vspace{.15cm}

\section{Final comments}

\begin{enumerate}[wide, labelindent=0pt]
\item Fix $n \geq 1$. A minimal factorization of the cycle $c_n=(1 \ 2 \ \cdots \ n) \in \mathbb S_{n}$ is a sequence $(\tau_1, \tau_2, \dots, \tau_{n-1})$ of $n-1$ transpositions satisfying $\tau_1 \tau_2 \cdots \tau_{n-1} = c_n$. Dénes \cite{Denes} showed that the number of minimal factorizations of the long cycle $c_n$ is given by Cayley's tree formula,  $n^{n - 2}$. Several bijections mapping minimal factorizations to trees and parking functions are known,  see \cite{KimSeo2003, Biane_J-V, IrvingRattan2020, Biane2002, Stanley_parking, BernardiMorales2013, Denes} and the references therein. 
In light of Theorems \ref{thm:maximal_chains_forests} and \ref{thm:maximal_chains_parking_functions}, a natural question arises: is there a sensible way to map a maximal chain of the priority lattice to a minimal cycle factorization? Doing so would require finding the proper analogs of priority forest and Jordan-Hölder permutation in this setting.  
This direction is the subject of ongoing work.\\

\item 
Let $P$ be a finite graded poset of rank $n$ with $\hat{0}$ and $\hat{1}$ and with rank function $\rho$.
The flag $f$-vector of $P$ is defined  as the function that assigns to 
$S\subseteq [n-1]$  the number $\alpha_P(S)$ of chains  
$
\hat{0} = t_0 < t_1 < \cdots < t_s = \hat{1}
$
of $P$ such that
$
S = \{\rho(t_1), \rho(t_2), \dots, \rho(t_{s-1})\}.
$ The flag $h$-vector of $P$ is the function that assigns to $S \subseteq [n-1]$,  the integer $\beta_P(S)$ defined as 
$\beta_P(S) = \sum_{T \subseteq S} (-1)^{|S - T|} \alpha_P(T).$ 
Define the \emph{flag quasi-symmetric function} as
\[
\mathcal{F}_P
=
\sum_{S \subseteq [n-1]} \beta_P(S) Q_S.
\]
where $Q_S$ denotes the Gessel's quasisymmetric function. It would be very interesting to understand the fkag quasi-symmetric function $\mathcal{F}_{\Pi(n)}$.
Note that since $\Pi(2)$ is a symmetric poset,
$\mathcal{F}_{\Pi(2)}$ is a symmetric function, but this is not the case for $n\ge2.$ 
We add some  data computed with the help of Sagemath: 
\begin{align*}
&\mathcal{F}_{\Pi(2)} = Q_{[1,2]}+Q_{[2,1]}+Q_{[3]}.
\end{align*}
\begin{align*}
&\mathcal{F}_{\Pi(3)} = Q_{[1,1,2]} + 3\,Q_{[1,2,1]} + 2\,Q_{[1,3]} + 4\,Q_{[2,2]} + 5\,Q_{[3,1]} + Q_{[4]}.
\end{align*}
\begin{multline*}
\mathcal{F}_{\Pi(4)} = Q_{[1,1,1,2]} + 6\,Q_{[1,1,2,1]} + 3\,Q_{[1,1,3]} + Q_{[1,2,1,1]} + 13\,Q_{[1,2,2]} 
+ 23\,Q_{[1,3,1]} \\+ 3\,Q_{[1,4]} 
 +  \ 6\,Q_{[2,1,2]}  + 21\,Q_{[2,2,1]} + 8\,Q_{[2,3]} 
+ Q_{[3,1,1]} + 15\,Q_{[3,2]} + 23\,Q_{[4,1]} + Q_{[5]}.
\end{multline*}
Afterwards, the expression becomes too messy to print, but can be easily recovered with our SageMath implementation, \url{https://github.com/adrianlillo/priority_lattice}.\\

\item Observe that the principal ideal appearing Figure \ref{fig:Interval_Pi3}  is isomorphic to $\Pi(2)$ (compare with Figure \ref{fig:L2}). We ask:
   How many principal ideals of  $\Pi(n)$ are isomorphic to $\Pi(m)$, for some $1\le  m\le n$? 
   Let $\gamma_n$ denote the number of principal ideals of $\Pi(n)$ that are isomorphic to $\Pi(m)$ with $m \le n$.
   The sequence $(\gamma_n)_{n\geq 1}$ begins as  
    $
    2, 4, 8, 14, 22 ,32, 44, \dots
    $
    This sequence seems to be  \url{https://oeis.org/A014206}, which is given by the formula  \[\gamma_n = n^2 + n + 2.\] 
We could also modify this question and ask the number $\theta_n$ of principal filters of $\Pi(n)$ (this is, intervals of the form $[P, \hat 1]$) that are isomorphic to $\Pi(m)$ with $m \leq n$. This sequence begins as 
$
2, 4, 10, 34
,154 ,874, \dots
$
and seems to be \url{https://oeis.org/A003422}, the sequence of left factorials, given by 
\[
\theta_n = \sum_{i=0}^{n-1} k!.
\]

\item
A \emph{sublattice dismantling} is a subdivision of a lattice into two non-empty sublattices. A lattice is sublattice dismantlable if it can be decomposed into 1-element lattices by consecutive sublattice dismantlings. We have checked with the help of SageMath that $\Pi(n)$ is sublattice dismantlable  for  $n=$ 1 to 5. Is this always the case?

\end{enumerate}

\section*{Acknowledgements}
We thank Stefan Trandafir and Mathieu Josuat-Vergnes for fruitful  conversations and an inspiring exchange of ideas.  The second author would like to thank the IRIF of the Université de Paris Cité for their hospitality, and  the opportunity to conduct part of this research at their facilities. 
This work has been partially supported by Grants PID2020-117843GB-
I00 and PID2024-157173NB-I00 funded by MCIN/AEI/10.13039/501100011033 and by FEDER, UE.

\printbibliography

\end{document}